\documentclass[10pt]{article}
\usepackage{color}
\usepackage{amssymb}
\usepackage{amsthm,array,amssymb,amscd,amsfonts,latexsym, url}
\usepackage{amsmath}
\usepackage[all]{xy}
\newtheorem{theo}{Theorem}[section]
\newtheorem{prop}[theo]{Proposition}
\newtheorem{claim}[theo]{Claim}
\newtheorem{lemm}[theo]{Lemma}
\newtheorem{coro}[theo]{Corollary}
\newtheorem{rema}[theo]{Remark}
\newtheorem{Defi}[theo]{Definition}
\newtheorem{ex}[theo]{Example}

\newtheorem{question}[theo]{Question}

\title{Unboundedness of zero-cycles on  higher dimensional Fano manifolds}
\author{Claire Voisin\footnote{The author is supported by the ERC Synergy Grant HyperK (Grant agreement No. 854361).}
}
\date{}

\newfont{\gothic}{eufb10}
\begin{document}

\maketitle
\begin{abstract} We show that, unlike del Pezzo surfaces, higher dimensional Fano manifolds do not in general satisfy  boundedness properties for their  ${\rm CH}_0$-group of $0$-cycles. For example, for quartic threefolds having a point of odd degree, there is no ``Coray type" upper bound on the minimal odd degrees of points. Also, the universal  ${\rm CH}_0$-group of Fano hypersurfaces    can  be ``unbounded'' (a notion which is related to infinite dimensionality in the sense of Mumford),  meaning that there is no integer $N$ such that $0$-cycles of degree at least $N$ are effective.
 \end{abstract}
\section{Introduction}
\subsection{Presentation of the results \label{secpresresults}}
In this paper, a variety defined over a field $K$ is  reduced and geometrically irreducible, that is, irreducible over $\overline{K}$. In particular, the Fano varieties, including del Pezzo surfaces, that we will consider, are smooth projective, geometrically irreducible, with ample anticanonical bundle. For a projective variety $X$ over a field $K$, and for a finite extension $K\subset L$,    the degree of a closed point $x$ of $X$  is defined to be  the degree of the field extension $K\subset L$, where $L$ is the residue field of $x$. This is also the degree of $x$ seen as a $0$-cycle of $X$ over $K$. It is expected by a conjecture of Cassels and Swinnerton-Dyer that any smooth cubic hypersurface with a closed  point of degree coprime to $3$ has a $K$-point.
In \cite{colliot} and \cite{voisindelpezzo}, the following results on points on del Pezzo surfaces of degrees $2$ and $3$ are  proved, improving a classical result of Coray \cite{coray}.
\begin{theo}\label{theocorayvoisin}  (i) (Colliot-Th\'{e}l\`{e}ne  \cite{colliot})  Let $S$ be a del Pezzo surface of degree $2$ over a field $K$ of characteristic $0$. Then if $S$ has a $0$-cycle of degree $1$ (equivalently, if $S$ has a closed point of odd degree), $S$ has a closed point of degree $1$, $3$ or $7$.

(ii) (Voisin \cite{voisindelpezzo})   Let $S$ be a smooth cubic surface defined over a field $K$ of characteristic $0$. If $S$ has a $0$-cycle of degree $1$ (equivalently, if $S$ has a closed point of degree coprime to $3$), $S$ has a closed  point of degree $1$ or  $4$.
\end{theo}
 Coray's result    in \cite{coray} was a weaker version of statement (ii), allowing also points  of degree $10$. Concerning statement (i), Koll\'ar and Mella noticed in \cite{kollarmella} that there are del Pezzo surfaces of degree $2$ with a point of degree  $3$ but no $K$-point (see \cite[Remarque 4.3]{colliot} for a complete decription of their argument).

Coming back to the case of cubic hypersurfaces, the following weak version of Cassels-Swinnerton-Dyer conjecture asking  for the existence of higher dimensional Coray-type bounds seems to be open.
\begin{question} \label{questioncoary} Does there exist a number $N(n)$ such that any   smooth cubic hypersurface $X$ of dimension $n$ over a field $K$ of characteristic $0$ having a $0$-cycle of degree $1$ has  a  closed point of degree at most  $ N(n)$ coprime to $3$?
\end{question}
The number $N(n)$ in this question should be of geometric origin, and not depend on the field  $K$, as in Theorem \ref{theocorayvoisin}.

 Colliot-Th\'{e}l\`{e}ne proved more generally in \cite{colliot} that  Coray type bounds as in Theorem \ref{theocorayvoisin} exist for any class of del Pezzo surfaces. It is thus  natural to ask   Question \ref{questioncoary} for higher dimensional  Fano varieties, as suggested in \cite{colliot}.  One  natural extension of Question \ref{questioncoary} is the following
\begin{question} \label{questioncoaryhighdeg} Does there exist a number $N(n)$ such that any   smooth Fano degree $d$ hypersurface $X\subset \mathbb{P}^{n+1}$ over a field $K$ of characteristic $0$ having  a $0$-cycle of degree $1$ has   closed points   of degree $d_i\leq N(n)$ that are globally coprime?
\end{question}
A weaker form of this  question is whether such a bound $N(n)$ exists depending on $X$ defined over $K$ but  working for any $X_L$, for all field extensions $K\subset L$.
Our first main result in this paper is a negative answer to (the weak form of) Question \ref{questioncoaryhighdeg}.
\begin{theo}\label{theovrai} (i)  There  exist  a field $K$ of characteristic $0$ and a quartic hypersurface $X\subset \mathbb{P}^4$ defined over $K$, with the following property: for any odd integer $N\geq7$, there exists an overfield $ L_N\supseteq K$ such that $X_{L_N}$ has  a closed point of degree $N$ and no closed point of odd degree strictly smaller than $N$.

(ii) The same statement holds for quartic or sextic  double solids instead of quartic hypersurfaces of  dimension $3$, for quartic hypersurfaces in $\mathbb{P}^5$ and for double covers of $\mathbb{P}^4$ ramified along a sextic or a octic hypersurface.
\end{theo}

 \begin{rema}\label{remastrongertheo} {\rm   The hypersurface in Theorem \ref{theovrai} will be the generic  hypersurface $X_{\eta}$, which is defined over the function  field $K$ of $\mathbb{P}(H^0(\mathbb{P}^4_\mathbb{C},\mathcal{O}_{\mathbb{P}^4_\mathbb{C}}(4)))$, as the generic fiber of the universal family over $B$. Here $\mathbb{C}$ could be replaced by any field of characteristic $0$. The field $L_N$  will be obtained by adding universally a point of degree $N$ (see paragraph \ref{secgenpointsNintro}). These fields have big transcendence degrees over $\mathbb{C}$.  }
 \end{rema}

Theorem \ref{theovrai} says in particular that a Coray type bound does not exist for Fano 3-folds or 4-folds  in general. Indeed, the hypersurfaces $X_{L_N}$ in Theorem \ref{theovrai} have index $1$, since they have $0$-cycles of degree $4$ and closed points of odd degree, while the index cannot be computed using points of {\it a priori} bounded degree.   Statement (ii) is more striking in view of Theorem \ref{theocorayvoisin}(i), since quartic double solids are Fano threefolds of index $2$ and their hyperplane sections are del Pezzo surfaces of degree $2$. Their geometry is very close to that of cubic threefolds, for which Question \ref{questioncoary} is still open.

 We will also study in this paper a   different  notion of (un)boundedness  for zero-cycles, which  can be seen as an arithmetic  analogue of Mumford's notion of  finite dimensionality for the ${\rm CH}_0$-group of a complex projective variety (see   \cite{mumford}).
 Consider  a smooth projective  variety $X$ defined over a field $K$ of characteristic $0$. Let us consider the following definition, inspired by Colliot-Th\'{e}l\`{e}ne's results  in \cite{colliot}:
\begin{Defi}\label{questionintro}  We will say that $X$ has bounded ${\rm CH}_0$  if  there exists a number $N$ (possibly depending on $X$ and $K$), such that for any field $L\supseteq K$, any $0$-cycle $z\in {\rm CH}_0(X_L)$  of  degree  ${\rm deg}\,z\geq N$ is effective (that is, rationally equivalent to an effective $0$-cycle).
\end{Defi}
We will say that $X$ has unbounded ${\rm CH}_0$  if an integer $N$ as above does not exist.
We note that  smooth projective (geometrically irreducible) curves $X$ have  bounded ${\rm CH}_0$, as this follows from the Riemann-Roch theorem. As a consequence,  a variety $X$ has bounded  ${\rm CH}_0$  if ${\rm CH}_0(X)$ is  universally representable, in the sense  that there exists a geometrically irreducible curve $C\stackrel{j}{\hookrightarrow} X$ such that for any field $L\supseteq K$, the morphism
$$j_*:{\rm CH}_0(C_L)\rightarrow  {\rm CH}_0(X_L)$$
is surjective.

One should note that for varieties defined over the complex numbers,  the notion of (un)boundedness introduced above is close to, but different from, the notion of (in)finite dimensionality   used  by Mumford in \cite{mumford}. Indeed,  a complex projective variety $X$ has finite dimensional ${\rm CH}_0$-group in Mumford's sense if there exists an integer $N$ such that any $0$-cycle $z\in {\rm CH}_0(X)$  of  degree  ${\rm deg}\,z\geq N$ is effective.  The definitions of ``boundedness'' and ``finite dimensionality'' are thus formally the same, but in Definition  \ref{questionintro},  we work with the universal ${\rm CH}_0$-group, namely the data of all groups ${\rm CH}_0(X_L)$ for $L\supseteq K$, while Mumford works with $K=\mathbb{C}$ and ${\rm CH}_0(X)$. For a complex projective variety, the difference between its  universal ${\rm CH}_0$-group and its ${\rm CH}_0$-group is an important  phenomenon whose study  led to   important recent progress on the study of stable rationality of smooth projective varieties. Indeed, there are now many  Fano varieties (hypersurfaces, complete intersections, cyclic double covers) over $\mathbb{C}$, which are known to have a nontrivial universal ${\rm CH}_0$ (see \cite{voisininvent}, \cite{CTPi}, \cite{totaro} and Schreieder's survey \cite{schreiedericm} for more references), while  their ${\rm CH}_0$-group is trivial (all points are rationally equivalent), because they are  rationally connected.

Inspired by the work of Colliot-Th\'{e}l\`{e}ne  \cite{colliot}, we  proved in \cite{voisindelpezzo} the following
\begin{theo} \label{theorappeldp}   Let $S$ be a smooth del Pezzo surface of degree $d_S\leq 3$, where $d_S:={\rm deg}\,c_1(K_S)^2$,  defined over a field of characteristic $0$. Then any $0$-cycle $z\in {\rm CH}_0(S)$ of degree  $\geq 21$  is effective. In particular, $S$ has   bounded  ${\rm CH}_0$.
\end{theo}
Colliot-Th\'{e}l\`{e}ne establishes Theorem \ref{theorappeldp} for rational  surfaces over  a field $K$ of  characteristic $0$ having  a $K$-point. Theorem C of  \cite{CTCoray}  states that  surfaces fibered in genus $0$ curves over $\mathbb{P}^1_K$ have bounded ${\rm CH}_0$. In that case, one can take for $N$ half the number of singular fibers.
Applying the same method as in \cite{voisindelpezzo}, we will extend  for completeness in   Section \ref{seccomplement}   Theorem  \ref{theorappeldp} to del Pezzo surfaces of degree $d_S$, with  $4\leq d_S\leq 9$.

\begin{theo} \label{theocompletdp}  There exists an integer $N$ with the following property: for any    smooth del Pezzo surface $S$   defined over a field $K$ of characteristic $0$,  any $0$-cycle $z\in {\rm CH}_0(S)$ of degree $\geq N$  is effective. In particular, $S$ has   bounded  ${\rm CH}_0$.
\end{theo}
Theorem \ref{theocompletdp} proves boundedness of   the ${\rm CH}_0$ group of  del Pezzo surfaces with a number $N$ which  is even independent of $S$ and the ground field $K$.
The proof of  Theorem \ref{theocompletdp} will give  an explicit estimate  for the integer $N$.  In the case of smooth projective geometrically integral curves $X$, the number $N$ can be taken to be  the genus of $X$, hence it depends only on the deformation class of $X$.
\begin{rema}{\rm  As Colliot-Th\'{e}l\`{e}ne mentioned to me, for del Pezzo surfaces of degree $\geq 5$, it is possible to  give another proof of the fact  that they have   bounded ${\rm CH}_0$, building on many beautiful  known results on the arithmetic of del Pezzo surfaces of degree $\geq 5$ (see \cite{VA}, \cite{caoyang}, \cite{krashen}, \cite{chernousov}). It is likely that one can even prove that they are  uniformly bounded as above, with a possibly  better bound than the one obtained here.}
\end{rema}

When $X$ is a Fano variety over a field $K$ of characteristic $0$, $X$ is geometrically rationally connected, so we have ${\rm CH}_0(X_{\overline{K}})=\mathbb{Z}$, because any two points $x,\,y\in X(\overline{K})$ belong to  a rational curve in $X$ defined over $\overline{K}$ (see \cite{komimo}). The Chow group
${\rm CH}_0(X)_0$ of $0$-cycles of degree $0$ is a torsion group. More precisely, there exists an integer $M$ (depending on $X$ and the field $K$) such that
for any overfield  $L\supseteq K$, ${\rm CH}_0(X_L)_{0}$ is of $M$-torsion. This statement observed in \cite[Proposition 11]{CTinvent2005} follows either  from a decomposition of the diagonal argument as in \cite{blochsri}, or directly using  the rational connectedness, that provides over a finite extension   $K'\supseteq K $ of degree $d$,  a dominant rational map
$$\phi: Y\times \mathbb{P}^1\dashrightarrow X$$
 mapping $Y\times\{0\}$ to a given point $x\in X(K')$, and such that the restriction of $\phi$ to $Y\times \{\infty\}$ is dominant of degree $N$. Then by a trace argument, for any field extension $L\supseteq K$,  any $0$-cycle of degree $0$ on $X_L$ is annihilated  by $M:=Nd$ in ${\rm CH}_0(X_L)$.

At the opposite, when  $S$ is a  surface over $\mathbb{C}$ with $p_g(S)\not=0$, or more generally any smooth projective variety $X$ over $\mathbb{C}$ with $H^{l,0}(X)\not=0$ for some $l\geq 2$,  Mumford's  celebrated theorem \cite{mumford}, later generalized by Roitman \cite{roitman}, is   the following
\begin{theo}\label{theomumford} (\cite{mumford}, \cite{roitman}) Let $X$ be a smooth projective variety over $\mathbb{C}$ such that  $$H^0(X,\Omega_{X/\mathbb{C}}^l)\not=0\,\,{\rm for \,\,some}\,\, l\geq 2.$$ Then, choosing a  point $x\in X(\mathbb{C})$, for any $N\geq 0$, the natural map
$$X^{(N)}(\mathbb{C})\rightarrow {\rm CH}_0(X)_{0}/{\rm Torsion},$$
$$ z\mapsto z-N x$$
is not surjective.
\end{theo}
In the theorem above,  $X^{(N)}$ denotes the $N$-th symmetric product of $X$.
The conclusion  says that for any integer $N\geq 0$, there exists a cycle
$z_0\in {\rm CH}_0(X_\mathbb{C})_{0}$ such that the $0$-cycle $z_0+N x\in {\rm CH}_0(X_\mathbb{C})$, which is  of degree $N$, is not effective modulo torsion. This  shows that $X$ has has infinite dimensional ${\rm CH}_0$-group in the  Mumford sense, hence { a fortiori}  unbounded  ${\rm CH}_0$  in the sense of Definition \ref{questionintro}.

In  the case where $X$ is a smooth  Fano variety defined over a field $K$ of characteristic $0$, we know that $H^0(X,\Omega_{X/K}^l)=0$ for any $l> 0$, and furthermore we know as explained above that ${\rm CH}_0(X)_0$ is universally a torsion group, so a priori we cannot directly use Mumford's strategy.
Our second main result is the following:
\begin{theo}\label{theomain} Assume $d$ is even, $n\geq 3$ and $d\geq 2\lceil\frac{n+2}{3}\rceil$. Then the very general degree $d$ hypersurface
$X\subset \mathbb{P}^{n+1}$  defined over $\mathbb{C}$ has unbounded  ${\rm CH}_0$ in the sense of Definition \ref{questionintro}.

There exist smooth Fano degree $d$ hypersurfaces $X\subset \mathbb{P}^{n+1}$ defined over $\overline{\mathbb{Q}}$, with unbounded  ${\rm CH}_0$.

The same results hold for  double covers of $\mathbb{P}^n$, $n\geq 3$, ramified along a smooth hypersurface of degree $2d\geq n+1$.
\end{theo}

The proof of Theorem \ref{theomain} will be obtained by combining the arguments of Mumford \cite{mumford} and Totaro \cite{totaro}.  Totaro proved in \cite{totaro} that very general  hypersurfaces as in Theorem \ref{theomain} do not  have universally trivial ${\rm CH}_0$. This is proved by using  Koll\'ar specialization \cite{kollar} in characteristic $2$, and by proving that the existence of a nonzero algebraic form of degree $>0$ on the desingularized central  fiber  $\widetilde{X}_0$ prevents $\widetilde{X}_0$ to have universally trivial ${\rm CH}_0$.  The specialization method then says that the geometric generic fiber does not have universally trivial ${\rm CH}_0$. We combine these arguments with the fact that, more precisely,  the desingularized central  fiber  $\widetilde{X}_0$ has   a nonzero algebraic form of degree $>1$. This in turn will prevent $\widetilde{X}_0$ to have bounded ${\rm CH}_0$ thanks to Theorem \ref{theocrit} and  Corollary \ref{corofautilleprouver} below, inspired by Mumford's theorem.  We will start by reviewing  Mumford's argument in Section \ref{secmumford} and we will prove  the following version of it, which works as well in  nonzero characteristic. For any smooth projective variety $X$ of dimension $n$ over a field $K$, and any positive integer $N$, we denote by $X^{(N)}_0$ the Zariski open set of $X^{(N)}$ parameterizing unordered sets of $N$ distinct points in $X$ and by $\Gamma_N\subset X^{(N)}_0\times X$ the universal subscheme. For any integer $m$, let $L_m=K(X^{(m)}_0)$, so that  $X_{L_{m}}$ is the generic fiber of the projection $X^{(m)}_0\times X\rightarrow X^{(m)}_0$.
\begin{theo}\label{theocrit} Assume there exists an effective cycle $Z\in{\rm CH}^n(X^{(N)}_0\times X)$ with the property that the degree of $Z$ over $X^{(N)}_0$ is   $N'<N$, and a $0$-cycle $z_0\in {\rm CH}_0(X)$  of degree $N-N'$, such   that, over the generic point $\eta$ of $X^{(N)}_0$, one has
\begin{eqnarray}\label{eqgammaZx} \gamma_{N}:=\Gamma_{N,\eta}=Z_\eta+z_{0,L_N}\,\,{\rm in}\,\,{\rm CH}_0(X_{L_{N}}).
\end{eqnarray}
Then  one has $H^0(X,\Omega_{X/K}^l)=0$, assuming we are in one  of the following situations.
\begin{enumerate} \item \label{itsit1} One has $l\geq 2$ and ${\rm char}\,K=0$.
\item \label{itsit2} One has $l= 2$ and ${\rm char}\,K$ is arbitrary.
\item \label{itsit3} One has $l=3$, ${\rm char}\,K=2$ and the dimension of $X$ is $\leq 5$.
\item \label{itsit4} One has $l\geq 2$, ${\rm char}\,K$  is arbitrary and $N\geq\frac{ N'(n+1)}{l}$.
\end{enumerate}
\end{theo}
\begin{rema}{\rm The case \ref{itsit1} of characteristic $0$   is due to Mumford \cite{mumford} and Roitman \cite{roitman}.

 Case \ref{itsit3} may seem special but it will be needed  for the proof of Theorem \ref{theovrai}.
Case \ref{itsit4} is the one used for the proof of theorem \ref{theomain}.}
\end{rema}
As a corollary of Theorem \ref{theocrit},\ref{itsit4}, we get the following analogue of Mumford's theorem \ref{theomumford}.
\begin{coro}\label{corofautilleprouver}  Let $X$ be a smooth projective variety over any field $K$. If $X$ has bounded ${\rm CH}_0$, then $h^0(X,\Omega_{X/K}^l)=0$ for $l\geq 2$.
\end{coro}
\begin{proof} Indeed, if $X$  has bounded ${\rm CH}_0$, there exists a number $N'$ such that for any overfield $L\supseteq K$, any $0$-cycle $z\in{\rm CH}_0(X_L)$ of degree at least $N'$ is effective. Choose $z_0\in{\rm CH}_0(X)$ of degree $d>0$, and let $N:=N'+k d$. Let $L:=K(X^{(N)}_0)$. The $0$-cycle $\Gamma_{N,\eta}-kz_{0,\eta}\in {\rm CH}_0(X_L)$ is of degree $N'$, hence effective, thus producing an effective  cycle $$Z'\in{\rm CH}^n(X^{(N)}_0\times X)$$ as above, of degree $N'$ over $X^{(N)}_0$, and satisfying (\ref{eqgammaZx}). Letting $k$ tend to infinity,  $N$ gets  arbitrarily large while  $N'$ is fixed, so  we can apply Theorem \ref{theocrit},\ref{itsit4}.
\end{proof}

\subsection{Overview of the methods and organization of the paper}
\subsubsection{Degree $N$ generic point and boundedness \label{secgenpointsNintro}}
For any projective variety $X$ over a field $K$, one  says that $X$ has universally trivial ${\rm CH}_0$ if $X$ has a $0$-cycle $z$  of degree $1$ and for any overfield $L\supseteq K$, ${\rm CH}_0(X_L)=\mathbb{Z}z_L$. It is  proved in \cite{AuCTPa} that if $X$ is smooth, $X$ has universally trivial ${\rm CH}_0$ if and only if  the generic point $\gamma_1=\delta_{X}\in X(L_1)$ is rationally equivalent to $z_{L_1}$ on $X_{L_1}$. That this condition is necessary is clear since both have degree $1$ on $X_{L_1}$. In the other direction, the rational equivalence $z_{L_1}=\delta_X$ in ${\rm CH}_0(X_{L_1})$ provides a Chow decomposition of the diagonal for $X$, hence for  $X_L$, for any $L\supseteq K$, and this in turn implies that ${\rm CH}_0(X_L)=\mathbb{Z}z_L$, see \cite{AuCTPa}.

For the study of the boundedness of ${\rm CH}_0$, we will consider as in section  \ref{secpresresults} the fields $L_N=K(X^{(N)}_0)$ and the degree $N$  generic point $\gamma_N\in X(L_N) $ defined as the generic fiber of the first projection $\Gamma_N\rightarrow X^{(N)}_0$. Let $h_X\in {\rm CH}_0(X)$ be of degree $d>0$. For  a  hypersurface $X$ of degree $d$, we will take $h_X=c_1(\mathcal{O}_X(1))^n\in {\rm CH}_0(X)$.  One obvious necessary condition for $X$ to have bounded ${\rm CH}_0$ is that there exists an integer $N'$ such that for any $k\geq 0$, the $0$-cycle $\gamma_{N}-k h_{X_{L_{N}}}$ of degree $N'$ is effective on $X_{L_{N}}$, where $N=N'+kd$.  Let us note that this  criterion   is almost  optimal since we have the following
\begin{lemm}\label{lemmanew16juillet}  Assume ${\rm char}\,K=0$ and $X$ is a  smooth projective variety with  $i(X)=1$. Choose  a $0$-cycle  $z\in{\rm CH}_0(X)$ which is of degree $1$.  Then $X$ has bounded ${\rm CH}_0$ if and only if there exists an integer $N$ such that for any $k\geq 0$, $\gamma_{N+k}-k z_{L_{N+k}}$ is effective on $X_{L_{N+k}}$.
\end{lemm}
\begin{proof} The ``only if'' is obvious from the definition. In the other direction, we consider any  field $K'$ containing $K$, and observe that if  $\gamma_{N+k}-k z_{L_{N+k}}$ is effective on $X_{L_{N+k}}$, then $\gamma'_{N+k}-k z_{L'_{N+k}}$ is effective on $X'_{L'_{N+k}}$, where we use the notation $X':=X_{K'}$, $L'_{N+k}=K'({X'}^{(N)}_0)$. This implies by a specialization argument that

 (*) {\it for any effective $0$-cycle $w$ of degree $N+k$ on $X'$, $w-k z_{L'}$ is effective on $X'$.}

 Indeed, by  \cite[Lemme 2.10]{colliot}, Fulton's specialization of cycles \cite[10.3]{fulton} preserves effectivity.

Finally, we write any $0$-cycle on $X'$ as
\begin{eqnarray}\label{eqnarray15juillet}  z'=z'_1-z'_2\,\,{\rm in}\,\,{\rm CH}_0(X')
\end{eqnarray}
where $z'_1$ and $z'_2$ are effective.  By (*), we can rewrite (\ref{eqnarray15juillet}) as
\begin{eqnarray}\label{eqnarray15juilletbis}  z'=z''_1-z''_2+\lambda z_{L'}\,\,{\rm in}\,\,{\rm CH}_0(X')
\end{eqnarray}
for some $\lambda\in \mathbb{Z}$, where $z''_1$ and $z''_2$ are effective of degree $\leq N$.  The degree $1$ cycle $z_{L'}$ has the property that there exists an integer $M$ such that for any effective $0$-cycle $w\in {\rm CH}_0(X')$ of degree $\leq N$, the $0$-cycle
$M z_{L'}-w$ is effective, of degree $\leq M$. This can be seen by applying Riemann-Roch to a smooth curve in $X'$ supporting both $z_{L'}$ and $w$. We can thus rewrite  (\ref{eqnarray15juilletbis})
as
\begin{eqnarray}\label{eqnarray15juilletter}  z'=z''_1+z'''_2+\lambda' z_{L'}\,\,{\rm CH}_0(X')
\end{eqnarray}
where $z'''_2$ is effective of degree $\leq M$, and $z''_1$ is effective of degree $\leq N$. Thus , if ${\rm deg}\,z'\geq M'+N$, we have $\lambda'\geq 0$ and $z'$ is effective. So $X$ has bounded ${\rm CH}_0$.
\end{proof}

 We will use this  criterion  to prove unboundedness of ${\rm CH}_0$ in Theorem \ref{theomain}. Effectivity of $\gamma_{N}-k h_{X_{L_{N}}}$ gives an equality
\begin{eqnarray}\label{eqpourcycleeffectifintro}  \gamma_{N}=z_{N'}+ k h_{X_{L_{N}}}\,\,{\rm in}\,\, {\rm CH}_0(X_{L_{N}}),
\end{eqnarray}
where $z_{N'}$ is an effective $0$-cycle of degree $N'$ on $X_{L_{N}}$. We will translate (\ref{eqpourcycleeffectifintro})
 in geometric form,
which in turn gives, for some dense Zariski open subset $U\subset X^{(N)}_0\times X$,
\begin{eqnarray}\label{eqpourcycleeffectifintroetalee}  \Gamma_{N\mid U\times X}=Z_{N'}+ k U\times h_X\,\,{\rm in}\,\, {\rm CH}^n(U\times X),
\end{eqnarray}
where $Z_{N'}$ is effective of degree $N'$ over $U$.
 In these arguments, we restrict to the open set $X^{(N)}_0$ because it is smooth, so that we can take cycle classes and their action on cohomology or forms, which we will do in  next  paragraph.

\subsubsection{Adapting the original proof of Mumford's theorem \label{secintromu}}
In Section \ref{secmumford}, we will recall the original  Mumford argument for the proof of   Theorem \ref{theomumford} and adapt it to prove Theorem \ref{theocrit}.
 The modern proofs of Mumford's theorem  are as follows:  one proves first that, if ${\rm CH}_0(X)$ is finite dimensional, then there exists a smooth projective curve $j:C\hookrightarrow  X$ such that $j_*: {\rm CH}_0(C)\rightarrow {\rm CH}_0(X)$ is surjective. One then concludes by an argument \`{a} la Bloch-Srinivas: if $(C,j)$ satisfies the above property, then there exist a
positive integer $M$, and two  cycles, $\Gamma_C$ supported on $X\times C$  and  $\Gamma\in {\rm CH}^n(X\times X)$  supported on $D\times  X$ for some proper closed algebraic subset $D\subset X$, such that
\begin{eqnarray}\label{eqblsricurveintro} M\Delta_X=\Gamma_C+\Gamma \,\,{\rm in}\,\,{\rm CH}^n(X\times X),
\end{eqnarray}
where $\Delta_X\subset X\times X$ is the diagonal.
Taking  the  Betti cycle classes in (\ref{eqblsricurveintro}), and letting them act on cohomology by the formula $$\gamma^*(\eta)={\rm pr}_{1*}(\gamma\smile {\rm pr}_2^*\eta),$$ one concludes
that the class of  any $l$-form on $X$ vanishes  for $l\geq2$  on the open set $X\setminus D$, hence is identically $0$. Indeed, the class of the diagonal acts as the identity and the class $[\Gamma_C]$ acts trivially on forms of degree $\geq 2$ since its  action factors through restriction to $C$.

The proof sketched above is elegant but, due to the coefficient  $M$ on the left in (\ref{eqblsricurveintro}), it does not allow to conclude that some torsion invariant of $X$ vanish, and a fortiori it does not generalize in nonzero characteristic. One version of Mumford's theorem  that  works well for any smooth projective variety over any field $K$ (including nonzero  characteristic) is the fact that a decomposition of the diagonal
\begin{eqnarray}\label{eqblsricurveintrosasnM} \Delta_X=\Gamma_C+\Gamma \,\,{\rm in}\,\,{\rm CH}^n(X\times X),
\end{eqnarray}
with  $\Gamma_C$ supported on $X\times C$ and  $\Gamma $ supported on $D\times  X$ for some proper closed algebraic subset $D\subset X$, implies that $H^0(X,\Omega_{X/K}^l)=0$ for $l\geq 2$. This implication is obtained as in \cite{totaro} by using the Gros cycle classes
$\gamma^{n,n}\in H^n(X\times X,\Omega_{X\times X/K}^{n})$  for $\gamma\in {\rm CH}^n(X\times X)$ (see \cite{gros}) and their action on $H^0(X,\Omega_{X/K}^l)$.

However boundedness of ${\rm CH}_0(X)$  does not give a decomposition (\ref{eqblsricurveintrosasnM}) for some  curve $C\subset X$. This latter condition, that implies boundedness of ${\rm CH}_0$ as was mentioned in  \ref{secpresresults}, is more restrictive, as  can be  already seen in the case of del Pezzo surfaces. Instead, in order to prove Theorem \ref{theocrit} and Corollary \ref{corofautilleprouver}, we will adapt the original proof of Mumford's theorem. This  is a remarkable rank argument for traces of algebraic forms, which goes as follows. Assume ${\rm CH}_0(X)$ is finite dimensional, so that there exists an integer $N$ such that any $0$-cycle of degree $\geq N$ on $X$ is effective. Choose $x\in X(\mathbb{C})$. Then for any effective  $0$-cycle $ z\in X^{(N+1)}_0$ of degree $N+1$, $z-x$ is effective. Ignoring for simplicity the fact that an effective $0$-cycle $z'$ rationally equivalent to $z-x$ could have multiplicities, it  follows that there exists a multivalued rational map $X^{(N+1)}_0\dashrightarrow X^{(N)}_0$, that we write as
 \begin{eqnarray}\label{eqdiagram}\xymatrix{
&Y\ar[d]^{\phi}\ar[r]^{\psi}& X^{(N)}_0 &&\\
&X^{(N+1)}_0&&&
}
 \end{eqnarray}
 with the property that $\phi$ is dominant, and for any $y\in Y(\mathbb{C})$
 \begin{eqnarray}\label{eqmultYphipsi}\Gamma_{N+1,*}(\phi(y))=x+ \Gamma_{N,*}(\psi(y))\,\,{\rm in}\,\,{\rm CH}_0(X).
 \end{eqnarray}
 Equation (\ref{eqmultYphipsi}) induces then the equality
 \begin{eqnarray}\label{eqmeqformsrevision}\phi^*([\Gamma_{N+1}]^*\eta=\psi^*( [\Gamma_{N}]^*\eta)\,\,{\rm in}\,\,H^0(Y,\Omega_{Y/\mathbb{C}}^l).
 \end{eqnarray}

of forms, for any differential form $\eta\in H^0(X,\Omega_{X/\mathbb{C}}^l)$, with $l>0$. The vanishing $\eta=0$ for $l\geq 2$ follows by comparing the generic ranks of both sides, when $l\geq2$.

It turns out that the proof above extends with some work to the setting of boundedness over any field. The boundedness assumption is stronger than the finite dimensionality assumption, in the sense that it provides roughly speaking the diagramm (\ref{eqdiagram}), with $Y$ a Zariski open set of $X^{(N+1)}_0$, which prevents difficulties caused by the possible inseparability of the map $\phi$. Another difficulty appears when adapting the argument above, namely in nonzero characteristic, an effective $0$-cycle of degree $N$  on $X$ may not be given by an element of $X^{(N)}$. This means that the rational map $\psi $ of (\ref{eqdiagram}) has to be replaced by the data of any effective  cycle $Z\in{\rm CH}^n( Y\times X)$ of degree $N$ over $Y$. The irreducible closed algebraic subsets $Z_i$ of codimension $n$ appearing in $Z$  could be inseparable over $Y$, which makes  the rank estimate for the forms
$[Z]^*\eta$ more subtle than in nonzero characteristic (see Lemmas \ref{lepourranksepar} and \ref{letraceinsep}).

\subsubsection{Applying the specialization method \label{secintrospecmeth}}
In Section \ref{sectheomain}, we will give the proof of Theorems \ref{theomain} and \ref{theovrai}.  The proofs build up   on   Totaro's work  \cite{totaro}.
 Totaro  proves that, for $n\geq3$, a very general complex hypersurface of degree $d\geq \frac{2(n+2)}{3}$ in  $\mathbb{P}^{n+1}$ does not have universally trivial ${\rm CH}_0$.
 Totaro's paper built on  the Colliot-Th\'{e}l\`{e}ne-Pirutka specialization method in  \cite{CTPi},  itself inspired by  the degeneration method in \cite{voisininvent}. This specialization method was used in all three cases  to disprove the ${\rm CH}_0$-universal triviality (hence the stable rationality) of the very general (or geometric generic) fiber. However Totaro used a completely different criterion (namely the existence of a nonzero algebraic form of positive degree) to prove that the special fiber in his specialization does not have  universally trivial ${\rm CH}_0$, while \cite{voisininvent} and \cite{CTPi} used a specialization with nontrivial Brauer class on the desingularized special fiber $\widetilde{X}_0$. Totaro uses the same  specialization as Koll\'ar in  \cite{kollar}, for  which  $\widetilde{X}_0$   has nonzero algebraic forms of degree $n-1$.
 The Koll\'ar central fiber $X_0$ of the Koll\'ar-Mori specialization of a hypersurface of degree $d$ in $\mathbb{P}^{n+1}$ (say with $d$ even to make the singularities even simpler)  is mildly singular, so that the specialization method applies to it (see \cite{totaro}). Totaro deduces the non-triviality of the universal ${\rm CH}_0$-group of the desingularized central fiber, or equivalently the non-existence of a Chow decomposition of the diagonal, from the existence of a nonzero algebraic form of positive degree, by taking cycle classes as discussed in paragraph \ref{secintromu}, and considering their actions on algebraic forms. What we do is  closely related, except that we apply these considerations not only  to the generic point $\delta_X=\gamma_1$, but also to the higher degree generic points $\gamma_N$. Proposition \ref{newpropdu16juillet} states  the version of the specialization method that  we  apply in order to establish unboundedness of ${\rm CH}_0$ for the geometric generic fiber.
Thanks to our Theorem \ref{theocrit}, the existence of algebraic  forms of degree $\geq 2$ on the desingularized central fiber $\widetilde{X}_0$  implies   the non-effectivity of $0$-cycles of the form $\gamma_{N}-k h_{\widetilde{X_0},L_{N}}$  on  the desingularized special fiber $\widetilde{X}_0$, for $k$ large and $N'=N-kd$ fixed. The same is  then true for the geometric generic fiber.

\subsubsection{Extra construction needed   for the proof of Theorem \ref{theovrai}\label{secintroextratrick}}
The quartic threefold of Theorem \ref{theovrai} is the generic quartic threefold $X_\eta$ defined over the field $K=k(B)$, $B=\mathbb{P}(H^0(\mathbb{P}^4, \mathcal{O}_{\mathbb{P}^4}(4)))$ over any field $k$ of characteristic $0$ (for example $k=\mathbb{C}$). The fields $L_N$ for odd $N$ and the points of degree $N$ are the same as before (see paragraph \ref{secgenpointsNintro}), namely
$L_N=K(X_\eta^{(N)})$ and $\gamma_N$ is the generic point of degree $N$.

We want to show that $X_{\eta,L_N}$ has no closed point $z_{N'}$  of  odd degree $N'<N$.  In order to apply the machinery used to prove Theorem \ref{theomain}, namely specialization and ranks of pull-backs of  algebraic forms, we need to do two things.

The first step is to find a rational equivalence equivalence relation similar to (\ref{eqpourcycleeffectifintro}) between $\gamma_N$ and $z_{N'}$ in ${\rm CH}_0(X_{\eta,L_N})$. We were in fact  not able to achieve this because we could not compute ${\rm CH}_0(X_{\eta,L_N})$. In Claim \ref{claim}, we construct instead (for $N\geq 7$)   a birational copy $\sigma(X_\eta)$ of $X_\eta$ in $X^{(N)}_{\eta,0}$. Our construction is done in such a way that for any point $x$ of $X_\eta$ where $\sigma$ is defined, (eg the generic point of $X_\eta$ defined over $L_1$,)  the corresponding point $\sigma(x)$ of $ X^{(N)}_{\eta,0}$ provides a $0$-cycle of degree $N$ on $X_\eta$ which is rationally equivalent to $Nx$. It is important that, when $X_\eta$ specializes to $\overline{X}$ under the Koll\'ar specialization, this map specializes to
$\overline{\sigma} : \overline{X}\dashrightarrow \overline{X}^{(N)}_0$.

 Spreading $z_{N'}$ to a cycle $\mathcal{Z}_{N'}\in{\rm CH}^3(X^{(N)}_{\eta,0}\times X_\eta)$, and using the fact that  $X^{(N)}_{\eta,0}$ is smooth, we get the composed correspondences
$\Gamma_N\circ \sigma$ and  $\mathcal{Z}_{N'}\circ \sigma\in {\rm CH}^3(U\times X_\eta)$, defined at least  over the non-empty Zariski open set $U$ of $X_\eta$ where $\sigma$ is a morphism with  value in $X^{(N)}_{\eta,0}$.

The group  ${\rm CH}^3(U\times X_\eta)$ can be computed  (see Lemma \ref{prodiag}) and we conclude that the two  correspondences
$\Gamma_N\circ \sigma$ and  $\mathcal{Z}_{N'}\circ \sigma\in {\rm CH}^3(U\times X_\eta)$  satisfy (modulo $2$) a rational equivalence relation
$$\Gamma_N\circ \sigma=\mathcal{Z}_{N'}\circ \sigma+\lambda U\times h_X\,\,{\rm in}\,\, {\rm CH}^3(U\times  X_\eta)\otimes\mathbb{Z}/2\mathbb{Z}$$
for some integer $\lambda$.

We now specialize everything  to the desingularization $\widetilde{\overline{X}}$ of the   Koll\'ar specialization $\overline{X}$ of $X_\eta$ and get the similar  relation
\begin{eqnarray}\label{eqnewdu24}\Gamma_N\circ \overline{\sigma}={Z}_{N'}\circ \overline{\sigma}+\lambda U\times h_{\widetilde{\overline{X}}}\,\,{\rm in}\,\, {\rm CH}^3(\overline{U}\times \widetilde{\overline{X}})\otimes\mathbb{Z}/2\mathbb{Z},\end{eqnarray}
where ${Z}_{N'}$ is the Fulton specialization of $\mathcal{Z}_{N'}$, hence is effective.
As we are in characteristic $2$, this  gives, for  any $2$-form $\omega$ on $\widetilde{\overline{X}}$,
$$ \overline{\sigma}^*([\Gamma_N]^*\omega)=\overline{\sigma}^*([{Z}_{N'}]^*\omega)\,\,{\rm in}\,\,H^0(\widetilde{\overline{X}},\Omega_{\widetilde{\overline{X}}/K_0}^2).$$
Finally, we  prove that the restriction map $\overline{\sigma}^*$ from  algebraic $2$-forms on $\widetilde{\overline{X}}^{(N)}_0$ to algebraic $2$-forms on $\widetilde{\overline{X}}$  is injective. Combining these arguments, we conclude that we get a relation
of $2$-forms on $\widetilde{\overline{X}}^{(N)}_0$:
$$ \Gamma_N^*\eta=Z_{N'}^*\eta\,\,{\rm in}\,\, H^0(\widetilde{\overline{X}}^{(N)}_0,\Omega_{\widetilde{\overline{X}}^{(N)}_0/K_0}^2).$$
As $N'<N$ and $Z_{N'}$ is effective, we conclude to a contradiction by comparing the generic  ranks on both sides.

The restriction to dimensions $3$ and $4$ in Theorem \ref{theovrai} comes from the last argument. When computing the generic rank of these forms, we meet some difficulties in nonzero characteristic, due to the fact that the computation of the  generic rank of the trace of an algebraic form does not work in the same way in nonzero characteristic. However it works for degree $2$ forms in any characteristic, which we use in dimensjon $3$. For   degree $3$ forms in characteristic $2$, we have to  use the tensor rank instead  (see Lemma  \ref{letraceinsep}), and this solves the dimension 4 case.

\vspace{0.5cm}

{\bf Thanks.} {\it   I thank Laurent Manivel and Joseph Landsberg for answering my questions on the rank of alternate forms and  Jean-Louis Colliot-Th\'{e}l\`{e}ne for  useful  explanations concerning del Pezzo surfaces. I also thank Burt Totaro for his    careful reading and suggestions, and the referees for their work which greatly  improved the exposition.}

\section{\label{secmumford} A review and extension of Mumford's results on infinite dimensionality  of ${\rm CH}_0$}
Our goal in this section is to prove Theorem  \ref{theocrit}.
We first  review  Mumford's theorem, adapted  to work over any field.   Let $X$ be  a smooth projective variety of dimension $n$ over a field $K$,  and let $X^{(N)}_0\subset X^{(N)}$ be the  open subset parameterizing $N$ distinct unordered points, which is smooth. The universal subscheme
$$\Gamma_N\subset X^{(N)}_0\times X$$
 is flat of degree $N$ and proper  over $X^{(N)}_0$. We denote by
$${\rm pr}_{X^{(N)}_0}: X^{(N)}_0\times X\rightarrow X^{(N)}_0 ,\,\,{\rm pr}_{X}: X^{(N)}_0\times X\rightarrow X$$
the two projections.

As $X$ is smooth projective of dimension $n$, ${\rm pr}_{X^{(N)}_0}$ is proper and  there is for any $l$ a push-forward map
$${\rm pr}_{X^{(N)}_0*}: H^{n}(X^{(N)}_0\times X,\Omega_{X^{(N)}_0\times X/K}^{n+l})\rightarrow H^0(X^{(N)}_0,\Omega_{X^{(N)}_0/K}^l)$$
and thus,  for any class $\gamma\in H^n(X^{(N)}_0\times X,\Omega_{X^{(N)}_0\times X/K}^n)$, an action
\begin{eqnarray}\label{eqvanGammaUdolb} \gamma^*: H^0(X,\Omega_{X/K}^{l})\rightarrow H^0(X^{(N)}_0,\Omega_{X^{(N)}_0/K}^l),\\
\nonumber
\gamma^*(\alpha)={\rm pr}_{X^{(N)}_{0}*}(\gamma\wedge {\rm pr}_X^*\alpha).
\end{eqnarray}

In particular, for $\gamma=[\Gamma_N]$, where $[\Gamma_N]\in H^n(X^{(N)}_0\times X,\Omega_{X^{(N)}_0\times X/K}^n)$ is the Gros  cycle class used in \cite{totaro}, and which depends only on the class of $\Gamma_N$ in ${\rm CH}^n(X^{(N)}_0\times X )$, we get for any $\alpha\in H^0(X,\Omega_{X/K}^{l})$ an $l$-form
$$[\Gamma_N]^*\alpha\in H^0(X^{(N)}_0,\Omega_{X^{(N)}_0/K}^l).$$

\begin{rema}\label{rematrace} {\rm  The construction above gives more generally, for any $l$, for any smooth variety $Y$ over $K$ and any cycle $\Gamma\in{\rm CH}^n (Y\times X)$, a morphism
$$[\Gamma]^*: H^0(X,\Omega_{X/K}^l)\rightarrow  H^0(Y,\Omega_{Y/K}^l).$$
In practice, we will compute $[\Gamma]^*\omega$ as a trace, at least at the generic point of $Y$.}
\end{rema}
 Let  $q:X^N\rightarrow X^{(N)}$ be the natural quotient map, and let $X^N_0:=q^{-1}(X^{(N)}_0)\subset X^N$.
We observe next  that
$$ (q,{\rm Id}_X)^*\Gamma_N=\sum_i\Gamma_i\,\,{\rm in}\,\,{\rm CH}^n(X^N_0\times X),$$
where $\Gamma_i\subset X^N_0\times X$ is the graph of the $i$-th projection ${\rm pr}_i: X^N_0\rightarrow X$.
It follows that
\begin{eqnarray}\label{eqpareilpourcrchet} (q,{\rm Id}_X)^*[\Gamma_N]=\sum_i[\Gamma_i] \,\,{\rm in}\,\,H^n(X^N_0\times X,\Omega_{X^N_0\times X/K}^n).
\end{eqnarray}

The contravariant functoriality of the constructions above
gives
\begin{eqnarray}\label{eqtireU}
q^*([\Gamma_N]^*\alpha)=\sum_i[\Gamma_i]^*\alpha=\sum_i{\rm pr}_i^*\alpha \,\,{\rm in}\,\,H^0(X^{N}_0,\Omega_{X^{N}_0/K}^l)
\end{eqnarray}
for any
$\alpha\in H^0(X,\Omega_{X/K}^l)$.
We  now make the following assumption

\vspace{0.5cm}

 (*) {\it there exist  an effective cycle
$$Z\in {\rm CH}^n(X^{(N)}_0\times X)$$ of relative degree $N'<N$ over $X^{(N)}_0$
and a
Zariski open set $U\subset  X^{(N)}_0$
such that

(a)  the corresponding morphism
$$\phi_Z:U\rightarrow X^{(N')}$$
is well-defined and takes value in $X^{(N')}_0$ and,

(b)  for some $z_0\in {\rm CH}_0(X)$ of degree $N-N'$,}
\begin{eqnarray}\label{eqvanGammaU} \Gamma_{N\mid U\times X}= U\times z_0+\Gamma_{N'}\circ \phi_{Z}\,\,{\rm in}\,\,{\rm CH}^n(U\times X).\end{eqnarray}

\vspace{0.5cm}

The  cycle class
$$[\Gamma_N]\in H^n(X^{(N)}_0\times X,\Omega_{X^{(N)}_0\times X/K}^n)$$
then satisfies
\begin{eqnarray}\label{eqvanGammaUdolbclass} [\Gamma_{N}]_{\mid U\times X}=[U\times z_0]+(\phi_Z,{\rm Id}_X)^*[\Gamma_{N'}]\,\,{\rm in}\,\,H^n(U\times X,\Omega_{U\times X}^n).\end{eqnarray}
From (\ref{eqvanGammaUdolbclass}), we deduce that, under the assumption (\ref{eqvanGammaU})
\begin{eqnarray}\label{eqvanGammaUdolbenfinvan} ([\Gamma_{N}]^*\alpha)_{\mid U}=\phi_Z^*([\Gamma_{N'}]^*\alpha)\,\,{\rm in}\,\,H^0(U,\Omega_{U/K}^l)\end{eqnarray}
for any $l>0$ and $\alpha\in  H^0(X,\Omega_{X/K}^{l})$, because $[U\times z_0]^*$ acts as zero on $l$-forms for $l>0$.
Let $\widehat{U}\subset X^N$ be the inverse image of $U$ under the quotient map $q:X^N_0\rightarrow X^{(N)}_0$ and let $\hat{\phi}_Z:=\phi_Z\circ q:\widehat{U}\rightarrow X^{(N')}_0$.

Equation (\ref{eqvanGammaUdolbenfinvan}) gives after pull-back to $\widehat{U}$, using (\ref{eqtireU})

\begin{eqnarray}\label{eqvcontrad} \sum_{i=1}^N{\rm pr}_i^*\alpha=\hat{\phi}_Z^*([\Gamma_{N'}]^*\alpha) \,\,{\rm in}\,\,H^0(\widehat{U},\Omega_{\widehat{U}/K}^l)\end{eqnarray}
for any $l>0$ and $\alpha\in  H^0(X,\Omega_{X/K}^{l})$.

We now prove

\begin{prop}\label{theomuro} (A version of Mumford-Roitman theorem \cite{mumford}, \cite{roitman}.) If $X$ has dimension $n$ and  satisfies assumption (*), then
$H^0(X,\Omega_{X/K}^l)=0$ for  $l\geq 2$.
\end{prop}

\begin{proof}[Proof of Proposition  \ref{theomuro}]  We will use  the naive notion ${\rm rk}(\omega)$ of rank  of a $l$-form $\omega\in \bigwedge^l V^*$, for any finite dimensional  vector space $V$. Namely ${\rm rk}(\omega)$ is defined as the rank of
the linear map given by interior product
\begin{eqnarray}\label{eqrankomegatrad} \lrcorner \omega: V\rightarrow \bigwedge^{l-1}V^*,\end{eqnarray}
$$v\mapsto v \lrcorner \omega.$$

We have  the following rather obvious
\begin{lemm}\label{lerank} (i) Let $\phi: V\rightarrow W$ be a linear map, and let $\omega\in\bigwedge^l W^*$. Then  ${\rm rk}(\phi^*\omega)\leq {\rm rk}(\omega)$.

(ii) Let $V$ be a vector space of the form $V=V_1\bigoplus \ldots\bigoplus V_N$ and let $\omega_i\in\bigwedge^lV_i^*$. Then, if $l\geq 2$, the  rank of \begin{eqnarray}\label{eqpouromega} \omega=\sum_i{\rm pr}_i^*\omega_i\end{eqnarray} is given by
\begin{eqnarray}\label{eqrangnaif} {\rm rk}(\omega)=\sum_i{\rm rk}(\omega_i).\end{eqnarray}

(iii) The  rank of a $l$-form $\omega\in \bigwedge^lV^* $  is lower-semicontinuous on $ \bigwedge^lV^* $ equipped with the Zariski topology.

(iv) The rank of a form is subadditive, that is : ${\rm rk}(\sum_i\omega_i)\leq \sum_i{\rm rk} (\omega_i)$.
\end{lemm}
\begin{proof} (i) follows immediately from the definition of the  rank, and from the fact that
$$\lrcorner\phi^*\omega=\phi^*\circ \lrcorner\omega\circ \phi: V\rightarrow \bigwedge^{l-1} V^*.$$
For (ii),  a form $\omega$ as in (\ref{eqpouromega}) has  rank ${\rm rk}(\omega)\leq \sum_i{\rm rk}(\omega_i)$ because the map
$\lrcorner\omega$ can be written as \begin{eqnarray} \label{eqnarraynew11nov} \lrcorner\omega=\sum_i {\rm pr}_i^*\circ \lrcorner\omega_i\circ {\rm pr}_i.  \end{eqnarray}

To see that ${\rm rk}(\omega)$ is  actually  $ \sum_i{\rm rk}(\omega_i)$ when $l\geq 2$, let $V_\omega$ be the kernel of $\lrcorner\omega$.
 We observe from (\ref{eqnarraynew11nov}) that a vector $$v=v_1+\ldots+ v_N\in V=V_1\bigoplus \ldots\bigoplus V_N$$ belongs to $V_\omega$ if and only if \begin{eqnarray}\label{eqvanpourle}v_i\lrcorner\omega_i=0\,\,{\rm  for\,\, all}\,\, i. \end{eqnarray}
Note that this is at this point that we use the condition $l\geq 2$, which implies that each $v_i\lrcorner\omega_i$ has degree $l-1>0$, so that the spaces  ${\rm pr}_i^*(\bigwedge^{l-1}V_i^*)$ are in direct sum in $\bigwedge^{l-1}V^*$.
Equality (\ref{eqrangnaif}) immediately  follows.

(iii) The lower semi-continuity of ${\rm rk}(\omega)$ immediately follows from the lower semi-continuity of the rank of a linear map or matrix, since it is defined as the rank of linear map $\lrcorner\omega$ of (\ref{eqrankomegatrad}).

(iv) follows from (i) and (ii) when $l\geq 2$, and is trivial when $l=1$.
\end{proof}
\begin{rema}{\rm Lemma \ref{lerank}(ii) is completely wrong if $l=1$, as any $1$-form has rank $1$. This observation is at the core of Mumford's theorem.}\end{rema}

We now conclude the proof of Proposition  \ref{theomuro}.  We argue by contradiction and assume that there exists a nonzero $\alpha\in H^0(X,\Omega_{X/K}^l)$, with $l\geq 2$.
We observe that thanks to Lemma \ref{lerank}(iii), a nonzero $l$-form $\alpha$ on $X$ has a generic  rank ${\rm rk}_{\rm gen}(\alpha)$, and that at any point $x\in X$, the  rank of $\alpha$ is $\leq {\rm rk}_{\rm gen}(\alpha)$.
Formula (\ref{eqtireU}) for $N'$ and Lemma \ref{lerank},(i) and (ii), imply that the  rank of $\phi_Z^*([\Gamma_{N'}]^*\alpha)$  at the generic point of $X^{(N)}_0$ is not greater than $ N'{\rm rk}_{\rm gen}(\alpha)$. Similarly, as $l\geq2$, formula (\ref{eqtireU}) and  Lemma \ref{lerank}(ii) give that
\begin{eqnarray}\label{eqgenrankgamman2611}{\rm rk}_{\rm gen}([\Gamma_{N}]^*\alpha)=  N{\rm rk}_{\rm gen}(\alpha),
\end{eqnarray}
which contradicts (\ref{eqvanGammaUdolbenfinvan}) since $N'<N$. (Note that  we used above  the fact that $q$ is \'etale over $X^{(N)}_0$ so that $q^*$ preserves the rank,  and similarly for $N'$.)
\end{proof}

For the proof of Theorem \ref{theocrit} in the cases \ref{itsit2} and \ref{itsit3} of nonzero characteristic, we will not use however the naive rank of a form as above, but its tensor rank.
Let $V$ be a vector space and let
$\omega\in \bigwedge ^lV^*$. Following \cite{landsberg}, we define the tensor rank ${\rm trk}(\omega)$ of any nonzero $\omega\in\bigwedge^lV$ as the minimal integer $r$ such that
$$\omega=\omega_1+\ldots +\omega_r,$$
where each $\omega_i$ is decomposable.

We have  the following variant of Lemma \ref{lerank}.
\begin{lemm}\label{lerankprime} (i) Let $\phi: V\rightarrow W$ be a linear map, and let $\omega\in\bigwedge^l W^*$. Then ${\rm trk}(\phi^*\omega)\leq {\rm trk}(\omega)$.

(ii) Let $V$ be a vector space of the form $V=V_1\bigoplus \ldots\bigoplus V_N$ and let $\omega_i\in\bigwedge^lV_i^*$. Assume that ${\rm dim}\,V_i=l$ or ${\rm dim}\,V_i=l+1$. Then if $l\geq 2$, the tensor rank of \begin{eqnarray}\label{eqpouromegaprime} \omega=\sum_i{\rm pr}_i^*\omega_i\end{eqnarray} is equal to $N$, assuming $\omega_i\not=0$ for all $i$.

(iii)  In the situation of (ii), assume that either $l=2$ or ${\rm dim}\,V_i=l+2$, so that $\omega_i$ has an even rank $2r_i$ (hence tensor rank $r_i$). Then if $l\geq 2$  and  $ \omega=\sum_i{\rm pr}_i^*\omega_i$,   ${\rm trk}(\omega)=\sum_ir_i$.

(iv) The tensor rank of a $l$-form $\alpha\in \bigwedge^lV^* $ on  a vector space $V$ of dimension $n$ is lower-semicontinuous on $ \bigwedge^lV^* $ equipped with the Zariski topology, assuming $l\leq2$ or $l\geq n-2$.

(v) The tensor rank of a $l$-form is subadditive : ${\rm trk}(\sum_i\omega_i)\leq \sum_i{\rm trk}(\omega_i)$.
\end{lemm}

\begin{proof} (i) follows immediately from the definition of the tensor rank since the pull-back of a decomposable form is decomposable.

For (ii), as any $l$-form on a vector space of dimension $l$ or $l+1$ is decomposable, we get by definition of the tensor rank that a form $\omega$ as in (\ref{eqpouromegaprime}) has tensor rank $\leq N$.  Suppose it has tensor rank $<N$. Then it can be written as
$$\omega=\omega'_1+\ldots+\omega'_{N'}$$
with $N'<N$ and $\omega'_i$ decomposable. It follows that there is a vector subspace $V'_\omega\subset V$ of dimension $\geq {\rm dim}\,V-N'l$ contained in the kernel $V_\omega$ of $\omega$ consisting of vectors $v\in V$ such that $v\lrcorner\omega=0$, namely the space  $V'_\omega$  annihilated by all $\xi_{sr}\in V^*$, where
$\omega'_{s}=\xi_{s1}\wedge\ldots\wedge \xi_{sl}$.
However, for $\omega$ given by (\ref{eqpouromegaprime}), a vector $$v=v_1+\ldots+ v_N\in V=V_1\bigoplus \ldots\bigoplus V_N$$ belongs to the kernel  $V_\omega$ of $\omega$ if and only if \begin{eqnarray}\label{eqvanpourle1}v_i\lrcorner\omega_i=0\,\,{\rm  for\,\, all}\,\, i. \end{eqnarray}
(Note that, as before, this is at this point that we use the condition $l\geq 2$, which implies that each $v_i\lrcorner\omega_i$ has degree $l-1>0$.)
As all $\omega_i$ are nonzero, condition (\ref{eqvanpourle1}) defines a space of codimension $Nl$, which provides  a contradiction.

(iii) is proved similarly as (ii), except that the computation of the tensor rank is different. We use here the fact that when ${\rm dim}\,V_i=l+2$, $\bigwedge^lV_i^*\cong \bigwedge^2V_i$ so the tensor rank of $\omega_i$ is computed as half the rank of $\omega_i$ seen as a $2$-form on $V_i^*$.

(iv) When $n=l$ or $n=l+1$, the tensor rank of any $\omega\in \bigwedge^lV^*$ is $1$ if $\omega\not=0$ and $0$ otherwise. When $n=l+2$, it was already explained  that the tensor rank is half the rank of the corresponding element of $\bigwedge^2V$, and similarly when $l\leq 2$, hence in all these cases, the rank  is lower-semicontinuous.

(v) trivially follows from the definition of the tensor rank.
\end{proof}
\begin{rema}{\rm As before, Lemma \ref{lerankprime}(ii) is  wrong if $l=1$, as any nonzero $1$-form has  tensor rank $1$.}\end{rema}

\begin{rema}\label{reamsurlowersemicontpasbon} {\rm Lemma \ref{lerankprime}(iv) is wrong if $3\leq l\leq n-3$. For example, being of tensor rank $2$ means belonging to the proper secant variety of the Grassmannian ${\rm Grass}(l,n)$ in its Pl\"{u}cker embedding (that is, the points on a bisecant line of ${\rm Grass}(l,n)$). But the Zariski closure of the proper secant variety of the Grassmannian contains its tangential variety, where we get  elements of the form
$$\alpha_1\wedge\ldots\wedge \alpha_l+\sum_i\alpha_1\wedge\ldots\wedge\hat{\alpha_i}\wedge\ldots\wedge \alpha_l\wedge \beta_i,$$
which are of tensor rank $l$ if $n$ is large enough and $\beta_i$ are chosen in a general way.}\end{rema}

\begin{proof}[Proof of theorem \ref{theocrit}, \ref{itsit1} to \ref{itsit3}]  Proposition  \ref{theomuro}  proves almost directly Theorem \ref{theocrit} in characteristic $0$, that is \ref{itsit1}. Indeed,  under the assumptions of Theorem \ref{theocrit}, we have an  effective cycle $Z$ of relative degree $N'$  satisfying (\ref{eqgammaZx}) over the generic point of $X^{(N)}_0$, hence over a dense Zariski open set. In characteristic $0$, such a cycle $Z$ automatically provides a rational map $\phi$ to the symmetric product $X^{(N')}$. To avoid generically the singularities of $X^{(N')}$, one should in fact decompose the effective cycle $Z$ as a sum $\sum_i Z'_i$, where each $Z'_i$ is irreducible and has degree $N'_i$ over $X^{(N)}_0$, and $\sum_i N'_i=N'$. Then generically  the rational map $\phi_{i}:X^{(N)}\dashrightarrow X^{(N'_i)}$ induced by $Z'_i$ takes value in $X^{(N'_i)}_0$ and  (\ref{eqgammaZx}) can be rewritten as

  \begin{eqnarray}\label{eqfactphiZ} Z_{\mid U\times X}=\sum_i \Gamma_{N'_i}\circ \phi_i\,\,{\rm in}\,\,{\rm CH}^n(U\times X)\end{eqnarray} for a dense Zariski open set $U$ of $X^{(N)}_0$. The relation  (\ref{eqfactphiZ}) then implies the following variant
   \begin{eqnarray}\label{eqpourformetireeii} ([\Gamma_{N}]^*\alpha)_{\mid U}=\sum_i \phi_i^*([\Gamma_{N'_i}]^*\alpha)\,\,{\rm in}\,\,H^0(U,\Omega_{U/K}^l)\end{eqnarray}
   of (\ref{eqvanGammaUdolbenfinvan}),
for any $l>0$ and $\alpha\in  H^0(X,\Omega_{X/K}^{l})$.

We then get a contradiction if $N>N'$ and $l\geq 2$  by the same  rank argument as in the proof of Proposition \ref{theomuro}, using the subaddivity and lower continuity of the rank (Lemma \ref{lerank} (iii), (iv)). Indeed, if $r$ is the generic rank of $\alpha$, then the generic rank of the left hand side in (\ref{eqpourformetireeii}) is $Nr$, while the generic rank of the right hand side is $\leq N'r$.

In nonzero characteristic, nonseparability phenomena prevent, even at the generic point, to construct $\phi$ from $Z$  and to write the factorization (\ref{eqfactphiZ}), but still we can deduce from
(\ref{eqgammaZx}) the equality of forms
\begin{eqnarray}\label{eqtracedeformes} [\Gamma_N]^*\alpha=[Z]^*\alpha\,\,{\rm in}\,\, H^0(X^{(N)}_0,\Omega_{X^{(N)}_0/K}^l),
\end{eqnarray}
for any $\alpha\in H^0(X,\Omega_{X/K}^l)$, with $l>0$.
We now restrict equality (\ref{eqtracedeformes}) to the generic point $\eta={\rm Spec}\,K(X_0^{(N)})$ of $X^{(N)}_0$, which allows to compute the pull-backs of forms in (\ref{eqtracedeformes}) as traces, and argue by comparing the generic tensor ranks of both sides in (\ref{eqtracedeformes}).

  The effective cycle $Z\in{\rm CH}_0(X_\eta)$ is a sum of points $Z_i\in X(L_i)$ where each $L_i=K(Z_i)$ is a finite extension of $K(X^{(N)})$. Such an extension decomposes as
$$K(X^{(N)}_0)\subset L_{i,{\rm sep}}\subset L_i$$
where the first extension is separable and the second one is purely inseparable.

We now use the following
\begin{lemm}\label{letraceinsep} Let $ L_1\hookrightarrow L_2$ be a degree $d$ purely inseparable extension of fields containing a given field  $K$ of characteristic $p$.   Let  $\omega \in \bigwedge^l\Omega_{L_2/K}$ be a $l$-form.

(i) If $l=2$, and $0\not={\rm char}\,K$ is arbitrary,  ${\rm Tr}_{L_2/L_1}(\omega)$ has tensor rank $\leq 1$.

(ii) If $l=3$ and ${\rm char}\,K=2$,
\begin{eqnarray}\label{eqdesiredineqtrace}{\rm trk}({\rm Tr}_{L_2/L_1}(\omega))\leq d\,{\rm trk}(\omega).
\end{eqnarray}
\end{lemm}
\begin{proof} We can assume that $d=p$ and  $L_2= L_1[x]/(x^p-a)$, for some $a\in L_1$. For (i), we have  to prove that for any  $l$-form $\omega\in \Omega_{L_{2}/K}^2$, the tensor rank of ${\rm Tr}_{L_2/L_1}(\omega)$ is $\leq 1$.

For (ii), since $d=p=2$,  it suffices to prove that for any  decomposable $l$-form $\omega\in \Omega_{L_{2}/K}^l$, hence of tensor rank $1$, the tensor rank of ${\rm Tr}_{L_2/L_1}(\omega)$  is $\leq 2$ if $l=3$.

For any $l$, denote by $f^*:\Omega_{L_{1}/K}^l\rightarrow \Omega_{L_{2}/K}^l$ the natural $L_1$-linear map. The trace map on forms of degree $l\geq 1$ on $\bigwedge^l\Omega_{L_2/K}$ is explicitly described in \cite[0ADY, Lemma 2.2]{stacks-project}  as follows:

\begin{eqnarray}\label{eqtracemap}  {\rm Tr}(f^*\eta_1\wedge\ldots\wedge f^*\eta_{l-1}\wedge x^i dx)=0\,\,{\rm if}\,\,0\leq i\leq p-2\\
\nonumber
{\rm Tr}(f^*\eta_1\wedge\ldots\wedge f^*\eta_{l-1}\wedge x^{p-1} dx)=\eta_1\wedge\ldots\wedge\eta_{l-1}\wedge da\\
\nonumber
{\rm Tr}(\omega)=0\,\,{\rm if}\,\,\omega\in {\rm Im}(\Omega_{L_1/K}^l\otimes L_2\rightarrow \Omega_{L_2/K}^l).
\end{eqnarray}

From (\ref{eqtracemap}), we conclude that any trace
${\rm Tr}_{L_2/L_1}(\omega)$  can be written as $da\wedge\beta$ for some $\beta\in \Omega_{L_1/K}$. Hence if $l=2$, it has tensor rank $\leq 1$, proving (i).

We next assume $p=2$ and $l=3$. Let $\alpha,\,\beta,\,\gamma\in \Omega_{L_2/K}$. We can assume after changing the basis of the $3$-dimensional $L_1$-vector space $\langle \alpha,\,\beta,\,\gamma\rangle$  that
$$\alpha=f^*a_1+x f^*a_2+xdx,\,\beta=f^*b_1+xf^*b_2+ dx,\,\gamma=f^*c_1+xf^*c_2,$$
where $a_i,\,b_j,\,c_l\in \Omega_{L_1/K}$.
We then get
\begin{eqnarray}\label{eqwedgetrois} \alpha\wedge\beta\wedge \gamma=(xdx\wedge f^*b_1+a dx\wedge f^* b_2+f^*a_1\wedge dx+xf^*a_2\wedge dx)\wedge (f^*c_1+xf^*c_2)+ \zeta,
\end{eqnarray}
where $dx$ does not appear in $\zeta$. Applying the rules (\ref{eqtracemap}), we get (note that the signs do not matter here, as we are in characteristic $2$):

\begin{eqnarray}\label{eqwedgetroistrace} {\rm Tr}_{L_2/L_1}(\alpha\wedge\beta\wedge \gamma)\\
\nonumber
={\rm Tr}((xdx\wedge f^*b_1+a dx\wedge f^*b_2+f^*a_1\wedge dx+xf^*a_2\wedge dx)\wedge (f^*c_1+xf^*c_2))\\
\nonumber
= da\wedge b_1\wedge c_1+ada\wedge b_2\wedge c_2+da\wedge a_1\wedge c_2+da\wedge a_2\wedge c_1\\
\nonumber =da\wedge(b_1\wedge c_1+a b_2\wedge c_2+ a_1\wedge c_2+ a_2\wedge c_1 ).
\end{eqnarray}
The $2$-form $\eta:=b_1\wedge c_1+a b_2\wedge c_2+ a_1\wedge c_2+ a_2\wedge c_1$ can be written as
\begin{eqnarray}\label{eqeta}\eta=(b_1+a_2)\wedge c_1+(ab_2+a_1)\wedge c_2,\end{eqnarray} hence it has tensor rank $\leq 2$, which proves (ii).
\end{proof}
\begin{rema}{\rm Statement (ii) in Lemma \ref{letraceinsep} is wrong if we use the rank instead of the tensor rank (see Example \ref{example} below). This is the reason why we use the tensor rank in the proof of Theorem \ref{theocrit}, while the rank has better properties (see Remark \ref{reamsurlowersemicontpasbon}).}
\end{rema}
\begin{rema}{\rm Statement (ii) in Lemma \ref{letraceinsep} is wrong if $l\geq 4$ (see Example \ref{exampleell} below). This is the reason for the condition $l= 3$ in Theorem \ref{theocrit}, \ref{itsit3}, which itself is responsible for the limitation to dimension $\leq 4$ in Theorem \ref{theovrai}.}
\end{rema}
We now conclude the proof of Theorem \ref{theocrit} in cases \ref{itsit2} and \ref{itsit3}  as follows. We write over the generic point of $X^{(N)}_0$
$$Z=\sum_i Z_i,$$  where $Z_i\in X(L_i)$, $L_i:=K(Z_i)$ is a field  extension of $K(X^{(N)}_0)$ of  degree $N'_i$, and \begin{eqnarray}\label{eqpourNprimeni} N'=\sum_iN'_i.
\end{eqnarray}
The  field $L_i$  is  a purely inseparable extension $L_i\supseteq L_{i,{\rm sep}}$ of a separable extension
$$K(X^{(N)}_0)\subset L_{i,{\rm sep}}.$$
Let $f_i: Z_i\rightarrow X$ be the natural morphism over $K$. For a $l$-form $\alpha$ on $X$, the $l$-form
$[Z]^*\alpha$, taken at the generic point of $X^{(N)}_0$, can be written as
\begin{eqnarray}\label{eqtraceomegasepinsep}[Z]^*\alpha=\sum_i{\rm Tr}_{L_i/K(X^{(N)}_0)}( f_i^*\alpha)=\sum_i{\rm Tr}_{L_{i,{\rm sep}}/K(X^{(N)}_0)}({\rm Tr}_{L_i/L_{i,{\rm sep}}}( f_i^*\alpha)).\end{eqnarray}
Using our assumptions that either $l=2$ or $l\geq {\rm dim}\,X-2$, we get the inequality
\begin{eqnarray}\label{eqineq1} {\rm trk}_{\rm gen} (f_i^*(\alpha))\leq  {\rm trk}_{\rm gen}(\alpha),
\end{eqnarray} thanks to the semi-continuity statement (iv) in Lemma \ref{lerankprime}. In (\ref{eqineq1}), the generic tensor rank ${\rm trk}_{\rm gen}(\omega)$ of a form $\omega$ on a variety  is the tensor rank of $\omega $ at its generic point.

We now use our assumptions \ref{itsit2} and \ref{itsit3} in Theorem \ref{theocrit}. In both cases, by Lemma \ref{letraceinsep}, we get the inequalities
\begin{eqnarray}\label{eqineq2} {\rm trk}_{\rm gen} ({\rm Tr}_{L_i/L_{i,{\rm sep}}} (f_i^*\alpha))\leq{\rm deg}(L_i/L_{i,{\rm sep}}){\rm trk}_{\rm gen} ( f_i^*\alpha).
\end{eqnarray}

Finally, the separable trace ${\rm Tr}_{L_{i,{\rm sep}}/K(X^{(N)}_0)}$ works as in Proposition \ref{theomuro}. More precisely,  the following statement is almost  obvious.
\begin{lemm} \label{lepourranksepar} Let $L_1$ be a field over $K$ and $L_1\subset L_2$ be a finite separable field extension of degree $k'$. Then the trace
$${\rm Tr}_{L_2/L_1}: \Omega_{L_2/K}^l\rightarrow \Omega_{L_1/K}^l$$
has the property that for any $\omega\in \Omega_{L_2/K}^l$,
$$ {\rm trk}({\rm Tr}_{L_2/L_1}(\omega))\leq  k' {\rm trk}(\omega).$$
\end{lemm}
We now get from(\ref{eqineq2}) and Lemma \ref{lepourranksepar}
\begin{eqnarray}\label{eqineq3}{\rm trk} ({\rm Tr}_{L_{i}/K(X^{(N)}_0)}( f_i^*\alpha))= {\rm trk} ({\rm Tr}_{L_{i,{\rm sep}}/K(X^{(N)}_0)}({\rm Tr}_{L_i/L_{i,{\rm sep}}} (f_i^*\alpha)))
\\
\nonumber
\leq {\rm deg}(L_{i,{\rm sep}}/K(X^{(N)}_0)){\rm deg}(L_i/L_{i,{\rm sep}}){\rm trk}_{\rm gen} ( f_i^*\alpha)\\
\nonumber ={\rm deg}(L_i/K(X^{(N)}_0)){\rm trk}_{\rm gen} ( f_i^*\alpha)=N'_i{\rm trk}_{\rm gen} ( f_i^*\alpha)\leq N'_i{\rm trk}_{\rm gen} (\alpha).
\end{eqnarray}
By  (\ref{eqpourNprimeni}), (\ref{eqtraceomegasepinsep}) and subadditivity of the tensor rank (see Lemma \ref{lerankprime}(v)), we finally get under our assumptions
\begin{eqnarray}\label{eqineqfinproof} {\rm trk}_{\rm gen}( [Z]^*\alpha)\leq N' {\rm trk}_{\rm gen}(\alpha).
\end{eqnarray}
However, as already used previously, Lemma \ref{lerankprime}(ii) and (iii) implies
$$ {\rm trk}_{\rm gen}([\Gamma_N]^*\alpha)=
N{\rm trk}_{\rm gen}(\alpha),$$ because $l\geq 2$.
Inequality (\ref{eqineqfinproof}) and formula (\ref{eqtracedeformes}) thus imply that $\alpha=0$ since $N'<N$. Thus Theorem \ref{theocrit} is also  proved  in cases \ref{itsit2} and \ref{itsit3}.
\end{proof}

Let us give two examples concerning the   rank or tensor rank of an inseparable trace. We first give an example with $l=3$, where inequality (\ref{eqdesiredineqtrace}) is not satisfied if we replace the tensor rank by the rank.
\begin{ex} \label{example}  {\rm We consider as above a field extension $ L_1\subset L_2$, of the form  $x^2=a$, in characteristic $2$. Let $K\subset L_1$ be a subfield and let
$$V\subset \Omega_{L_2/K}$$ be a $6$-dimensional vector space. An element in $\bigwedge^3V\subset \Omega_{L_2/K}^3$ thus has rank $\leq 6$. We write now elements of
$V$ as \begin{eqnarray}\label{eqpouralpahadevel} \alpha=a_\alpha+x b_\alpha+ c_\alpha dx+ d_\alpha xdx,\end{eqnarray}
 where $a_\alpha,\,b_\alpha\in \Omega_{L_1/K},\,c_\alpha,\,d_\alpha\in L_1$. (Here we omit the notation $f^*$ and identify $\Omega_{L_1/K}/da$ and its image in $\Omega_{L_2/K}$.)  We assume that  $da,\,a_\alpha,\,b_\alpha$, $\alpha\in V$, generate a vector subspace of dimension $13$ of $\Omega_{L_1/K}$. We  claim that for adequately chosen
 $\omega\in  \bigwedge^3V$, one has
 $${\rm Tr}_{L_2/L_1}(\omega)=da\wedge \eta,$$
 where $\eta\in\bigwedge^2(\Omega_{L_1/K}/da)$ has rank $12$. Then the rank of ${\rm Tr}_{L_2/L_1}(\omega)$ is $13$, while the rank of $\omega$ is $6$. So we do not have the inequality
 $${\rm rk}({\rm Tr}_{L_2/L_1}(\omega))\leq 2{\rm rk}(\omega).$$
We now prove the claim. Let $W_{12}\subset \Omega_{L_1/K}$ be the $12$-dimensional vector   space generated by the $a_\alpha,\,b_\alpha$ introduced  above.  The $L_1$-vector space $W_{12}$ contains a $10$-dimensional vector subspace $W_{10}$, namely the space of those forms $\gamma$ for which the scalars $c_\gamma,\,d_\gamma\in L_1 $ of (\ref{eqpouralpahadevel}) vanish. From the computation (\ref{eqwedgetroistrace}), (\ref{eqeta}),
we get that for $\alpha,\,\beta,\,\gamma\in V_6$ with $\gamma\in W_{10}\subset W_{12}=V_6$, and $c_\alpha=0,\,d_\alpha=1$, $c_\beta=1,\,d_\beta=0$, one has
\begin{eqnarray}\label{eqwedgetroistracenouveau} {\rm Tr}_{L_2/L_1}(\alpha\wedge\beta\wedge \gamma)=da\wedge \eta .
\end{eqnarray}
where the $2$-form $\eta$ is given by
\begin{eqnarray}\label{eqeta1}\eta=a_\gamma\wedge(a_\beta+b_\alpha)+b_\gamma\wedge(ab_\beta+a_\alpha).\end{eqnarray}
It is clear that by taking $L_1$-linear combinations of forms $\eta$ as in (\ref{eqeta1}), one gets arbitrary elements in $W_{10}\wedge W_{12}\subset \bigwedge^2\Omega_{L_1/K}$, hence forms of rank $12$, even modulo $da$, which proves the claim.}
\end{ex}
Next we  give an example of  computation  of traces of $l$-forms, with $l\geq4$, where inequality (\ref{eqdesiredineqtrace}) is not satisfied.
\begin{ex}\label{exampleell} {\rm  With the same notation as in Example \ref{example}, let
$$\omega=\alpha_1\wedge\ldots\wedge \alpha_4.$$ This is an element of tensor rank $1$ in $\bigwedge^4\Omega_{L_2/K} $. Then we claim that, in general, the tensor rank of ${\rm Tr}_{L_2/L_1}(\omega)$ is at least $ 3$, so that we do not have the desired inequality ${\rm trk}({\rm Tr}_{L_2/L_1}(\omega))\leq 2{\rm trk}(\omega)$.
To prove the claim, we can assume up to changing the $L_1$ basis of $\langle \alpha_i\rangle_{L_1}$ that
$$\alpha_1=a_1+x b_1+  dx,\,\alpha_2=a_2+x b_2+  xdx,\,\alpha_3=a_3+x b_3,\,\,\alpha_4=a_4+x b_4,$$
where $a_i:=a_{\alpha_i},\,b_i:=b_{\alpha_i}\in \Omega_{L_1/K}$.
We thus get
$${\rm Tr}_{L_2/L_1}(\omega)={\rm Tr}_{L_2/L_1}(dx\wedge(a_2+x b_2)\wedge (a_3+x b_3)\wedge(a_4+x b_4))$$
$$+
{\rm Tr}_{L_2/L_1}(xdx\wedge (a_1+x b_1)\wedge (a_3+x b_3)\wedge(a_4+x b_4))$$
$$=da\wedge (b_2\wedge a_3\wedge a_4+ab_2\wedge b_3\wedge b_4+a_2\wedge a_3\wedge b_4+a_2\wedge b_3\wedge a_4)$$
$$+da\wedge(a_1\wedge a_3\wedge a_4+a(b_1\wedge b_3\wedge a_4+b_1\wedge b_4\wedge a_3+a_1\wedge b_3\wedge b_4)).$$

It remains to see that the tensor rank of the $3$-form
\begin{eqnarray}\nonumber \eta:= b_2\wedge a_3\wedge a_4+ab_2\wedge b_3\wedge b_4+a_2\wedge a_3\wedge b_4+a_2\wedge b_3\wedge a_4\\ \nonumber
+a_1\wedge a_3\wedge a_4+a(b_1\wedge b_3\wedge a_4+b_1\wedge b_4\wedge a_3+a_1\wedge b_3\wedge b_4)\in \bigwedge^3(\Omega_{L_1/K}/da)\end{eqnarray} modulo $da$ is at least $ 3$, when the forms $a_i$, $b_j$ are globally independent modulo $da$. Indeed, this implies that the tensor rank of $$da\wedge\eta={\rm Tr}_{L_2/L_1}(\omega)$$ is at least $3$. We have
$$\eta=a_3\wedge a_4\wedge(b_2+a_1) +ab_3\wedge b_4\wedge(b_2+a_1)+a_3\wedge b_4\wedge(a_2+ab_1 )+b_3\wedge a_4\wedge(a_2+aa_1),$$
$$=(a_3\wedge a_4+ab_3\wedge b_4)\wedge (b_2+a_1)+a_3\wedge b_4\wedge(a_2+ab_1 )+b_3\wedge a_4\wedge(a_2+aa_1).$$
We conclude that the rank of $\eta$ modulo $da$ is at least $ 7$. It follows  that its tensor rank modulo $da$ is at least $ 3$ when the forms $ a_i, \,b_j$ are globally independent modulo $da$.

}

\end{ex}

We finally complete the proof of Theorem \ref{theocrit} by treating   case \ref{itsit4}.
\begin{proof}[Proof of Theorem \ref{theocrit},\ref{itsit4}] We will use  the following result concerning the rank of inseparable traces.
\begin{lemm}\label{letraceinsepnew} Let $ L_1$ be a field containing a given field  $K$ of characteristic $p$, and let  $L_2= L_1[x]/(x^p-a)$, for some $a\in L_1$.   Let  $\omega \in \bigwedge^l\Omega_{L_2/K}$ be a $l$-form. We have

 \begin{eqnarray}\label{eqineqrktr} {\rm rk}({\rm Tr}_{L_2/L_1}(\omega))\leq p\,{\rm rk}(\omega)+1.\end{eqnarray}
\end{lemm}
\begin{proof} As before, for any $l$, we denote by $f^*:\Omega_{L_{1}/K}^l\rightarrow \Omega_{L_{2}/K}^l$ the natural $L_1$-linear map.
Let $r={\rm rk}(\omega)$, so  that there exists a $L_2$-vector subspace $W\subset \Omega_{L_2/K}$ of dimension $ r$ such that
$$\omega\in \bigwedge^lW\subset \bigwedge^l\Omega_{L_2/K}.$$
Let $\alpha_i,\,i=1,\ldots,\,r$ be a basis of $W$ over $L_2$. We write
$$\alpha_i=\sum_{j=0}^{p-1}x^j f^*a_{ij}+\sum_{j=0}^{p-1} b_{ij}x^jdx,$$
where $a_{ij}\in \Omega_{L_1/K}$, $b_{ij}\in L_1$. Let $V\subset \Omega_{L_1/K}$ be the vector subspace generated by the $a_{ij}$ and $da$.  Formula (\ref{eqtracemap}) shows that for any
$\omega\in \bigwedge^lW$,
$${\rm Tr}_{L_2/L_1}(\omega)\in \bigwedge^lV,$$
which proves (\ref{eqineqrktr}) since ${\rm dim}\,V\leq pr+1$.
\end{proof}
By writing any purely inseparable extension as a sequence of extensions as above, we get
\begin{coro} \label{corotrace} Let $L_1$  be a field containing a given field  $K$ of characteristic $p$, and let
$L_1\subset L_2$ be a finite  purely inseparable extension of degree $p^s$. Then for any  $\omega\in\bigwedge^l\Omega_{L_2/K}$, we have
 \begin{eqnarray}\label{eqineqtracelnew} {\rm rk}({\rm Tr}_{L_2/L_1}(\omega))\leq p^s {\rm rk}(\omega)+\frac{p^{s}-1}{p-1}<p^s ({\rm rk}(\omega)+1).
\end{eqnarray}
\end{coro}
Next  we have  the following obvious fact on  separable traces.
\begin{lemm} \label{lepourrankseparnaiverank} Let $L_1$ be a field over $K$ and $L_1\subset L_2$ be a finite separable field extension of degree $k'$. Then  for any $\omega\in \Omega_{L_2/K}^l$,
\begin{eqnarray}\label{eqineqtracesep2711} {\rm rk}({\rm Tr}_{L_2/L_1}(\omega))\leq  k' {\rm rk}(\omega).\end{eqnarray}
\end{lemm}

 Writing any field extension $L_1\subset L_2$ as a separable extension followed by a purely inseparable extension, we get, by combining inequality (\ref{eqineqtracelnew}) and (\ref{eqineqtracesep2711}), the following corollary:
 \begin{coro} \label{corotracetout2711} Let $L_1$  be a field containing a given field  $K$, and let
$L_1\subset L_2$ be a finite   extension of degree $d$. Then for any  $\omega\in\bigwedge^l\Omega_{L_2/K}$, we have
 \begin{eqnarray}\label{eqineqtracelnew2711} {\rm rk}({\rm Tr}_{L_2/L_1}(\omega))< d ({\rm rk}(\omega)+1).
\end{eqnarray}
\end{coro}
\begin{coro} \label{corotraceprime}  For any smooth variety $Y$ defined over a field $K$, any smooth projective variety $X$ over $K$ and any effective cycle
$Z\subset Y\times X$ of relative degree $N'$ over $Y$, for any algebraic form $\omega\in H^0(X,\Omega_{X/K}^l)$, we have
 \begin{eqnarray}\label{eqineqtracelnew26111508} {\rm rk}_{\rm gen}([Z]^*(\omega))< N' ({\rm rk}_{\rm gen}(\omega)+1).
\end{eqnarray}
\end{coro}
\begin{proof} We  decompose $Z$ as a sum of irreducible closed algebraic subsets $Z_i$ of $Y\times X$, and restrict over the generic point of $Y$. Then $Z_i$ provides a field extension $K(Y)\subset L_i$ of degree $N'_i$, with $\sum_i N'_i=N'$. We apply  inequality (\ref{eqineqtracelnew2711}) to each of these extensions and get that
\begin{eqnarray}\label{eqineqtracelnew26111514}{\rm rk}_{\rm gen}([Z]^*(\omega))\leq \sum_i  {\rm rk}_{\rm gen}([Z_i]^*(\omega))<
\sum_iN'_i ({\rm rk}_{\rm gen}(\omega) +1). \end{eqnarray}
The first inequality in (\ref{eqineqtracelnew26111514}) is by subadditivity in Lemma \ref{lerank}(iv), and for the second inequality, we use Lemma \ref{lerank}(i), lower-semicontinuity (Lemma \ref{lerank}(iii)) and (\ref{eqineqtracelnew2711}).
As $\sum_iN'_i= N'$, we get (\ref{eqineqtracelnew26111508}).
\end{proof}
We now fix the dimension $n$ of $X$ and observe that by definition, for any $l$-form $\omega$ on $X$,
$${\rm rk}_{\rm gen}(\omega)\leq n.$$
We thus get from (\ref{eqineqtracelnew26111514}) the inequality
\begin{eqnarray}\label{eqineqtracelnew26111533} {\rm rk}_{\rm gen}([Z]^*(\omega))< N'(n+1),
\end{eqnarray}
which is valid for any cycle $Z$ as above.
However, letting $Y:=X^{(N)}_0$, we know (see (\ref{eqgenrankgamman2611})) that
for any $\omega\in H^0(X,\Omega_{X/K}^l)$  with $l\geq 2$, we have
\begin{eqnarray}\label{eqineqtracelnew26111540} {\rm rk}_{\rm gen}(\Gamma_N^*(\omega))=N {\rm rk}_{\rm gen}(\omega) .
\end{eqnarray}
We conclude arguing as  in the previous proofs. As ${\rm rk}_{\rm gen}(\omega)\geq l$ if $\omega\not=0$, equality (\ref{eqtracedeformes}) and (\ref{eqineqtracelnew26111540}), (\ref{eqineqtracelnew26111533}) imply that $\omega=0$ for $l\geq 2$, once $N\geq \frac{N'(n+1)}{l}$.
\end{proof}
\section{Unboundedness of zero-cycles on Fano hypersurfaces of dimension $\geq 3$\label{sectheomain}}
We prove in this section Theorem \ref{theomain}.
The method is by specialization and uses the Totaro arguments in \cite{totaro}, combined with Theorem \ref{theocrit}.  Mori first observed in \cite{mori} that hypersurfaces of even degree  $2d$ specialize to double covers of hypersurfaces of degree $d$ ramified along a general member of $|\mathcal{O}(2d)|$. Koll\'ar specializes these cyclic double covers to characteristic $2$, where they become inseparable and mildly singular. He then proves that starting from dimension $3$,   they have desingularizations with a nonzero algebraic forms on the special fiber. This differential form is the desired obstruction to stable rationality used in \cite{totaro}.  We summarize here the results of \cite{kollar}, \cite{kollarbook} and \cite{totaro}, that will quickly lead us to  the proof of Theorem \ref{theomain}.

Let us recall the following definition  from \cite{CTPi}.
\begin{Defi} \label{defiCTPi} A
 projective morphism
$$\phi: Y\rightarrow X$$ between varieties defined over a field $K$ is universally ${\rm CH}_0$-trivial if for any
field $L \supseteq K$, the induced morphism

$$\phi_*:{\rm CH}_0(Y_{L})\rightarrow {\rm CH}_0(X_L)$$
is an isomorphism.
\end{Defi}
The hypersurfaces $X\subset \mathbb{P}^{n+1}$ we are going to consider are the same as in \cite{totaro}. They were first introduced  by Koll\'ar \cite{kollar}, but Totaro works in a slightly broader range of degrees.

\begin{theo} (Cf. \cite{kollar}, \cite[Section V.5]{kollarbook},  \cite{totaro}.) \label{theototaro} Assume $d$ is even, $n\geq 3$ and $d\geq 2\lceil\frac{n+2}{3}\rceil$. Then there exists a smooth degree-$d$ hypersurface $X\subset \mathbb{P}^{n+1}$ defined over a number field $K_0$, with a specialization $\overline{X}$ defined over a field $\mathbb{F}$ of characteristic $2$ (a finite extension of $\mathbb{F}_2$)  satisfying the following properties

1) The special fiber  $\overline{X}$ admits a desingularization
$\tau:\widetilde{\overline{X}}\rightarrow \overline{X}$ and the morphism $\tau$ is universally ${\rm CH}_0$-trivial.

2) The smooth projective variety $\widetilde{\overline{X}}$ admits a nonzero algebraic form
\begin{eqnarray}\label{eqomegasurXbartilde} 0\not=\omega\in H^0(\widetilde{\overline{X}},\Omega_{\widetilde{\overline{X}}/\mathbb{F}}^{n-1}).
\end{eqnarray}

\end{theo}
By specialization, we mean as usual that there is a flat projective morphism $\mathcal{X}\rightarrow {\rm Spec}\,R$, where $R$ is  d.v.r.  (in our case, a localization of a ring of  integers), whose special fiber is isomorphic to $\overline{X}$, and generic fiber is isomorphic to $X$.
\begin{rema}\label{remadouble} (Cf. \cite{CTPiIzvestya}.) {\rm The same result holds for cyclic double covers of $\mathbb{P}^n$, $n\geq 3$, ramified along a hypersurface of degree $2d$ with $2d\geq n+1$. Condition (2) above is proved in
\cite[Proposition 5.10]{kollarbook}.}
\end{rema}

\begin{proof}[Proof of Theorem \ref{theomain}]
Note that the conditions on $n$ and $d$ are the same in Theorem \ref{theomain} and Theorem \ref{theototaro}. Let $\mathcal{X}^0\rightarrow B^0$ be the universal smooth hypersurface. Using Lemma \ref{lemmanew16juillet}, one proves by standard arguments playing on the countability of the relative Chow varieties of $(\mathcal{X}^0)^{(N)/B}\times_{B^0} \mathcal{X}^0$ (see \cite{voisininvent}), that  the set of smooth complex hypersurfaces of degree $d$ and dimension $n$ having  bounded ${\rm CH}_0$  is a  countable union of closed algebraic subsets in the parameter space $B^0$, so in order to prove that the very general one has unbounded ${\rm CH}_0$, it suffices to prove that there exists one smooth hypersurface of degree $d$ and dimension $n$ with unbounded ${\rm CH}_0$.

 Let $X$ be as in Theorem  \ref{theototaro}. We  prove that ${\rm CH}_0(X)$ is not bounded. Assume by contradiction that for some integer $N'$, and for any overfield $L\supseteq K_0$, any $0$-cycle $z\in{\rm CH}_0(X_L)$ of degree $\geq N'$ is effective. For any integer $l\geq 0$, we take for $L$ the field $L_{N}=K_0(X^{(N)})$, $N:=N'+dk$. Then, as in paragraph \ref{secgenpointsNintro}, $X_{L_N}$ has  the  closed point $\gamma_{N}$, which is of degree $N$.
Recall the notation $h_X:=c_1(\mathcal{O}_X(1))^n\in {\rm CH}_0(X)$. We have ${\rm deg}\,h_X=d$, so
${\rm deg}\,(\gamma_{L_N}-kh_X)=N'$, and  our hypothesis implies that $\gamma_{L_N}-kh_X$ is effective of degree $N'$. This implies that there exists an effective  codimension $n$ cycle
$$T\subset X^{(N)}_0\times X$$
of degree $N'$ over $X^{(N)}_0$, such that for some dense Zariski open set $U\subset X^{(N)}_0$
\begin{eqnarray}\label{eqnpourrelzerocycleNd}
T_{\mid U\times X}=\Gamma_{N\mid U\times X}- k\, U\times h_X\,\,{\rm in}\,\,{\rm CH}^n(U\times X).
\end{eqnarray}
We now specialize these cycles to $\overline{X}$ and get an {\it effective}  specialized cycle $$\overline{T}\subset \overline{X}^{(N)}\times \overline{X}$$ and a dense Zariski open set $\overline{U}\subset  \overline{X}^{(N)}_0$ satisfying
\begin{eqnarray}\label{eqnpourrelzerocycleNdbar}
\overline{T}_{\mid \overline{U}\times \overline{X}}=\Gamma_{N\mid \overline{U}\times \overline{X}}- k\,\overline{U}\times h_{\overline{X}}\,\,{\rm in}\,\,{\rm CH}^n(\overline{U}\times \overline{X}).
\end{eqnarray}

 We observe next that a lift $\widetilde{\overline{T}}\subset \widetilde{\overline{X}}^{(N')}_0\times \widetilde{\overline{X}}$  of  the cycle $\overline{T}$  exists  as an {\it effective} cycle, at least over a dense  Zariski open set $U'\subset \widetilde{\overline{X}}^{(N')}_0$. Indeed the map $\tau$ is an isomorphism away from the (isolated) singular points of  $\widetilde{\overline{X}}$, so it suffices to choose the field $\mathbb{F}$ in such a way that there exists a rational point over each singular point of $\widetilde{\overline{X}}$. Thanks to the ${\rm CH}_0$ universal triviality of the desingularization map
$\tau$, stated in Theorem \ref{theototaro}(1), the effective cycle $\widetilde{\overline{T}}$  satisfies (maybe after shrinking $U'$) the same relation as in (\ref{eqnpourrelzerocycleNdbar}):
\begin{eqnarray}\label{eqnpourrelzerocycleNdbartilde}
\widetilde{\overline{T}}_{\mid U'\times \widetilde{\overline{X}}}=\Gamma_{N\mid U'\times \widetilde{\overline{X}}}- k \, U'\times h_{\widetilde{\overline{X}}}\,\,{\rm in}\,\,{\rm CH}^n(U'\times \widetilde{\overline{X}}).
\end{eqnarray}

As  $N=N'+dk$ can be taken arbitrarily large while $N'$ is fixed, and $\widetilde{\overline{X}}$ has a nonzero $n-1$-form  by Theorem \ref{theototaro}(2),  we get a contradiction with Theorem \ref{theocrit}, \ref{itsit4}, since $n\geq3$.
\end{proof}

The proof given above establishes more generally the following statement.

\begin{prop}\label{newpropdu16juillet} Let $\pi: X\rightarrow {\rm Spec}(R)$ be a flat projective morphism, where $R$ is a d.v.r with residue field $k_0$. Assume that the special fiber $\overline{X}$  admits a desingularization $\tau: \widetilde{\overline{X}}\rightarrow \overline{X}$ such that the following properties hold:
\begin{enumerate}
\item $\tau$ is universally ${\rm CH}_0$ trivial.
\item \label{item2du18jui} $\tau$ induces   a surjection on points over any field containing $k_0$.
\item  One has $H^0(\widetilde{\overline{X}},\Omega^l_{\widetilde{\overline{X}}/k_0})\not=0$ for some $l\geq 2$.
\end{enumerate}
Then the generic geometric fiber of $\pi$ has unbounded ${\rm CH}_0$.
\end{prop}
\begin{rema} {\rm Property  \ref{item2du18jui} can always be achieved after extending the field $k_0$ if $\widetilde{\overline{X}}$ has isolated singularities.}
\end{rema}
\subsection{Proof of Theorem \ref{theovrai}\label{sectheovrai}}

We focus in this section on the case of quartic threefolds and prove Theorem \ref{theovrai}. The case of quartic or sextic double solids  works exactly in the same way, using Remarks \ref{remadouble} and  \ref{rema2pourdouble}, \ref{rema3pourdouble} below. The cases of quartic fourfolds or  double covers of $\mathbb{P}^4$ ramified along a sextic or a octic hypersurface  also work well with a small  extra work needed to establish the analog of Lemma \ref{pro2form}.

In order to prove Theorem \ref{theovrai}, we will need the following two extra ingredients.
  Over any field $k$, let  $B=\mathbb{P}(H^0(\mathbb{P}^{n+1}_k,\mathcal{O}_{\mathbb{P}^{n+1}_k}(d)))$ and let \begin{eqnarray}\label{eqquarticuniv} \mathcal{X}\subset B\times \mathbb{P}^{n+1}_k\end{eqnarray} be the universal hypersurface, with morphism $\pi={\rm pr}_1: \mathcal{X}\rightarrow B$ and generic fiber
$X_{\eta}$ over $K:=k(B)$. As before,  for any hypersurface $X\subset \mathbb{P}^{n+1}$, we denote by $h_X\in{\rm CH}_0(X)$ the $0$-cycle $c_1(\mathcal{O}_X(1))^n$.
\begin{lemm}\label{prodiag} One has ${\rm CH}_0(X_{\eta})=\mathbb{Z} h_{X_{\eta}}$ and the image of ${\rm CH}^n(X_{\eta}\times X_{\eta})$ in $$\underset{{U}\underset{{\rm open}}{\subset} X_{\eta}}{\underrightarrow{\rm lim}}{\rm CH}^n(U\times X_{\eta})$$ is generated by $\Delta_{X_{\eta}}$ and ${\rm pr}_2^* h_{X_\eta}$.
\end{lemm}
\begin{proof} This statement appears in several places (see for example \cite{voisinkyoto}). We give a proof for completeness. The variety $\mathcal{X}$ admits the morphism
$$ f:={\rm pr}_2:\mathcal{X}\rightarrow \mathbb{P}^{n+1}.$$
The fiber of $f$ over any point $x$ of $ \mathbb{P}^{n+1}$ is the hyperplane $B_x=\mathbb{P}(H^0(\mathbb{P}^{n+1},\mathcal{I}_{x}(d)))\subset B$. Hence $\mathcal{X}$ is a projective bundle over $\mathbb{P}^{n+1}$ and it follows that ${\rm CH}^*(\mathcal{X})$ is generated as a ring  by $h_1:={\rm pr}_1^*c_1(\mathcal{O}_B(1))$ and $h_2:=f^*c_1(\mathcal{O}_{\mathbb{P}^{n+1}}(1))$. The restriction map
$${\rm CH}(\mathcal{X})\rightarrow  {\rm CH}(X_{\eta})$$
is surjective and annihilates $h_1^i,\,i>0$. It follows that ${\rm CH}(X_{\eta})$ is generated as a ring by  the restriction of $f^*h_2$, which proves the first statement  since $h_{X_{\eta}}=(h_2^n)_{\mid X_{\eta}}$.

For the second statement, we use a similar argument except that we consider now
\begin{eqnarray} \label{eqpourY}\mathcal{Y}:=\mathcal{X}\times_{B}\mathcal{X}.\end{eqnarray}
We  have  a surjection given by restriction over the generic point of $B$
$${\rm CH}(\mathcal{Y})\rightarrow {\rm CH}(X_{\eta}\times X_{\eta}).$$
As above, we consider now the morphism
$$g:=(f_1,f_2):\mathcal{Y}\rightarrow\mathbb{P}^{n+1}\times \mathbb{P}^{n+1}.$$
Above a point  $(x,y)$ of $ \mathbb{P}^{n+1}\times \mathbb{P}^{n+1}$ with $x\not=y$, the fiber of $g$ is a projective subspace $B_{x,y}\subset B$ of codimension $2$. (Above a point $(x,x)$ of $ \mathbb{P}^{n+1}\times \mathbb{P}^{n+1}$, the fiber of $g$ is the projective subspace $B_{x}\subset B$ of codimension $1$, so $g$ is a projective bundle only over $\mathbb{P}^{n+1}\times \mathbb{P}^{n+1}\setminus \Delta_{\mathbb{P}^{n+1}}$.)
Let $\Delta_{\mathcal{X}}\subset \mathcal{Y}$ be the diagonal. We have
$$\Delta_{\mathcal{X}}=g^{-1}(\Delta_{ \mathbb{P}^{n+1}}).$$
Denoting $\pi_2:\mathcal{Y}\rightarrow B$ the natural map, we see from the above analysis that ${\rm CH}(\mathcal{Y}\setminus \Delta_{\mathcal{X}}) $ is generated as a ring by
$\alpha={\rm \pi}_2^* c_1(\mathcal{O}_B(1))$,  $\beta=f_1^*c_1(\mathcal{O}_{\mathbb{P}^{n+1}}(1))$ and $\gamma=f_2^*c_1(\mathcal{O}_{\mathbb{P}^{n+1}}(1))$. It follows from the localization exact sequence that ${\rm CH}^n(\mathcal{Y})$ is generated as an abelian group by the class of $\Delta_{\mathcal{X}}$ and the monomials of degree $n$ in $\alpha,\,\beta,\,\gamma$.
When we restrict to $X_{\eta}\times X_{\eta}$, the powers $\alpha^i$  vanish for $i>0$. Similarly, when we restrict to
$U\times X_{\eta}$, where $U\subset X_{\eta}$ is a sufficiently small Zariski open subset, all powers $\beta^i$  vanish for $i>0$. Hence we conclude that for such an open set $U$,  ${\rm CH}^n(U\times X_{\eta})$ is generated as an abelian group by $\gamma^n$ and $\Delta_{X_{\eta}}$.
\end{proof}
\begin{rema}\label{rema2pourdouble}{\rm The same statement holds for the generic double cover of $\mathbb{P}^n$ ramified along a degree $2d$ hypersurface, and is proved similarly, as they  can be defined as a hypersurface in the total space of $\mathcal{O}(d)$   on $\mathbb{P}^n$.}
\end{rema}
Let $Y=\widetilde{\overline{X}}$ be the desingularization of a Koll\'ar-Mori specialization of a quartic threefold over  a field $\mathbb{F}$ of  characteristic $2$ considered in Theorem \ref{theototaro}. Recall that we denote by $Y^{(N)}_0$ the open set of the symmetric product $Y^{(N)}$ that parameterizes sets of $N$ distinct points (or smooth $0$-dimensional subschemes of length $N$) in $Y$.
\begin{lemm} \label{pro2form} For any $N\geq 1$, one has an isomorphism
$$H^0(Y^{(N)}_0,\Omega_{Y^{(N)}_0/\mathbb{F}}^2)\cong H^0(Y,\Omega_{Y/\mathbb{F}}^2)$$
given by the map $[\Gamma_N]^*$.
\end{lemm}
\begin{proof} The pull-back to $Y^{N}_0$ of $2$-forms on $Y^{(N)}_0$ is injective and maps $H^0(Y^{(N)}_0,\Omega_{Y^{(N)}_0/\mathbb{F}}^2)$ to the space of invariant $2$-forms under the symmetric group action. We now prove the following
\begin{claim}\label{claim1forms} The threefold $Y$ being as above, we have  $H^0(Y,\Omega_{Y/\mathbb{F}})=0$.
\end{claim}
\begin{proof} The variety $Y=\widetilde{\overline{X}}$ is a desingularization of the variety $\overline{X}$, obtained as  a cyclic double cover $r: \overline{X}\rightarrow Q$ of a quadric $Q\subset \mathbb{P}^4$ ramified along a smooth surface $S\subset Q$ which is a member of the linear system $|\mathcal{O}_Q(4)|$ defined by a section $s\in H^0(Q,\mathcal{O}_Q(4))$.
The variety $\overline{X}$ has isolated singular points over the vanishing locus $V(ds)$ of the section $ds\in H^0(Q,\Omega_{Q/\mathbb{F}}(4))$ constructed in \cite[Definition-Lemma V.5.4]{kollarbook}. It clearly suffices to show that
$$H^0(\overline{X}_0,\Omega_{\overline{X}_0/\mathbb{F}})=0,$$
where $ \overline{X}_0$ is the smooth locus of $\overline{X}$.
 We now restrict to
$ \overline{X}_0$ the exact sequence  established in \cite[Lemma V.5.5]{kollarbook} using  the characteristic $2$ assumption,  and get
\begin{eqnarray}\label{eqexcatkollar} 0\rightarrow r^*\mathcal{F}\rightarrow \Omega_{\overline{X}_0/\mathbb{F}}\rightarrow r^*\mathcal{O}_{Q_0}(-2)\rightarrow 0,
\end{eqnarray}
where $Q_0:=Q\setminus V(ds)$ and $\mathcal{F}$ is the rank $2$ vector bundle on $Q_0$ defined as
\begin{eqnarray}\label{eqdefF} \mathcal{F}:={\rm Coker}(ds: \mathcal{O}_{Q_0}(-4)\rightarrow\Omega_{Q_0/\mathbb{F}}).\end{eqnarray}
We have $H^0(\overline{X}_0,r^*\mathcal{O}_{Q_0}(-2))=0$, so by (\ref{eqexcatkollar})  we only have to prove that $H^0(\overline{X}_0,r^*\mathcal{F})=0$, which is clearly equivalent to $H^0(Q_0,\mathcal{F})=0$.
By Definition (\ref{eqdefF}) of $\mathcal{F}$, we have an exact sequence on $Q_0$
\begin{eqnarray}\label{eqexcatkollarF} 0\rightarrow \mathcal{O}_{Q_0}(-4)\rightarrow \Omega_{Q_0/\mathbb{F}}\rightarrow \mathcal{F}\rightarrow 0.
\end{eqnarray}
As $Q\setminus Q_0$ has codimension $3$ in $Q$, we have $H^1(Q_0,\mathcal{O}_{Q_0}(-4))=H^1(Q,\mathcal{O}_{Q}(-4))=0$. Thus $H^0(Q_0,\mathcal{F})=0$ follows from $H^0(Q_0,\Omega_{Q_0/\mathbb{F}})=0$.
\end{proof}
\begin{rema}\label{rema3pourdouble}{\rm The same statement holds for the double covers of $\mathbb{P}^3$,  replacing in the proof above the quadric $Q$ by $\mathbb{P}^n$.}
\end{rema}
Claim \ref{claim1forms} implies  that
$$H^0(Y^{N}_0,\Omega_{Y^{N}_0/\mathbb{F}}^2)=\bigoplus_{i}{\rm pr}_i^*H^0(Y,\Omega_{Y/\mathbb{F}}^2),$$
from which Lemma \ref{pro2form} follows by taking invariants under the symmetric group $\mathfrak{S}_N$.
\end{proof}
\begin{proof}[Proof of Theorem \ref{theovrai}]  The field $K$ and the quartic hypersurface $X$ are defined as follows. Working over any field of characteristic $0$, for example $\mathbb{Q}$ (or $\mathbb{C}$), let as before
$B=\mathbb{P}(H^0(\mathbb{P}^4,\mathcal{O}_{\mathbb{P}^4}(4)))$ and $ \mathcal{X}\rightarrow B$
 be the universal quartic hypersurface. The quartic hypersurface we consider is the generic fiber $X_\eta$ of the fibration $ \mathcal{X}\rightarrow B$. For any integer $N>0$, let $\mathcal{X}^{(N)/B}$ be the relative $N$-th symmetric product. The field $L_N$ is defined  as
\begin{eqnarray}\label{eqpourK} L_N:=\mathbb{Q}(\mathcal{X}^{(N)/B}).
\end{eqnarray}
The quartic $X_{L_N}$ is the generic fiber of the morphism
$${\rm pr}_1: \mathcal{X}^{(N)/B}\times_B\mathcal{X}\rightarrow \mathcal{X}^{(N)/B}.$$
The universal cycle
$$\Gamma_N\subset \mathcal{X}^{(N)/B}\times_B\mathcal{X}$$
provides as in paragraph \ref{secgenpointsNintro} a closed point $\gamma_{N}$ of $X_{L_N}$, which is of degree $N$.
Theorem \ref{theovrai} follows now from Proposition \ref{profinal} proved below.
\end{proof}
\begin{prop} \label{profinal} Assume that $N\geq 7$ is odd. Then the quartic threefold $X_{L_N}$ has no closed point of odd degree $M<N$.
\end{prop}
\begin{rema}{\rm The assumption $N\geq 7$ does  not seem to be a serious one. It will be used   only to simplify a detail of proof.}
\end{rema}
\begin{proof}[Proof of Proposition \ref{profinal}]  We argue by contradiction and assume that there is an effective $0$-cycle $z_M\in{\rm CH}_0(X_{L_N})$ of odd degree $M<N$. Let $\mathcal{Z}_M\in {\rm CH}^3(\mathcal{X}^{(N/B)}\times_B\mathcal{X})$ be any effective cycle whose restriction over the generic point of $\mathcal{X}^{(N/B)}$ is $z_M$. Our  first goal is to compare, at least modulo $2$, the two $0$-cycles
$$\gamma_{N},\,\,z_M\in {\rm CH}_0(X_{L_N}),$$    and get a relation as in (\ref{eqgammaZx}).
 We are not, in fact,  able to do this because we do not know how to compute  ${\rm CH}_0(X_{L_N})$, but Claim \ref{claim} and Corollary \ref{eqpourrestdiagcomp} will give us a partial comparison.
Consider  the generic  quartic threefold $X_\eta$ over the field $K:=\mathbb{Q}(B)$.  The symmetric product $X_{\eta}^{(N)}$ is singular along the complement
 $X_{\eta}^{(N)}\setminus  X_{\eta,0}^{(N)}$ and a fortiori  very singular along the small diagonal
$$\delta(X_{\eta})\subset X_{\eta}^{(N)},$$
where $\delta$ is the diagonal inclusion. Let us establish the following
\begin{claim} \label{claim} There exists a
rational map
$$\sigma: X_{\eta} \dashrightarrow X_{\eta}^{(N)}$$
with the following properties :

\begin{enumerate}
\item \label{item119juill} The image of $\sigma$ is not contained in the singular locus $X_{\eta}^{(N)}\setminus  X_{\eta,0}^{(N)}$ (so generically the image is contained in $X_{\eta,0}^{(N)}$).
    \item   We have the equality  \begin{eqnarray}\label{item219juill}\Gamma_N\circ \sigma= N\Delta_{X_\eta}\,\,{\rm in} \,\,\underset{{U}\underset{{\rm open}}{\subset} X_{\eta}}{\underrightarrow{\rm lim}}\,{\rm CH}^n(U\times X_\eta).
        \end{eqnarray}
        \end{enumerate}
\end{claim}
In (\ref{item219juill}), the notation $\Gamma_N\circ \sigma$ has the following meaning: using condition \ref{item119juill} and the fact that $X_{\eta,0}^{(N)}$ is smooth,  we can compose  the correspondences  $\Gamma_N$  and $Z_M:=\mathcal{Z}_{M,\eta}$ with  the  map $\sigma$   over the dense  Zariski open set $U$ of $X_\eta$ where $\sigma$ is  a well-defined morphism with  values in $X^{(N)}_0$. This gives us the  correspondences
 \begin{eqnarray} \label{eqtrucdu26juill}\Gamma_N\circ \sigma:=(\sigma,Id_{X_\eta})^*\Gamma_N,\,\,{Z}_{M}\circ \sigma:=(\sigma,Id_{X_\eta})^*{Z}_{M}\,\,\in{\rm CH}^n(U\times X_\eta).\end{eqnarray}
 \begin{rema}{\rm  Even if ${Z}_{M}$ is effective, the composition ${Z}_{M}\circ \sigma \in{\rm CH}^n(U\times X_\eta)$ might not be an effective cycle, but we do not need  effectivity of ${Z}_{M}\circ \sigma$  in the argument below.}
 \end{rema}
\begin{proof}[Proof of Claim \ref{claim}]
As we are in characteristic $0$, for any point $x$ of $X_\eta$ defined over a field $L$, we can choose by Bertini a smooth  plane section $C_x\subset X_{\eta,L}$ passing through $x$. The curve $C_x$ is thus of genus $3$. The point $x$ of $C_x$ defines a degree $1$ divisor  $D_x$ of $C_x$, and as $N\geq 7$, the linear system $|N D_x|$ is very ample. Thus, again by Bertini, there is a reduced  divisor $z_x$ of $C_x$ which is linearly equivalent to $ND_x$ in $C_x$. As $C_x\subset X_{\eta,L}$, this provides as well a degree $N$ effective reduced $0$-dimensional subscheme $z_x$ of $X_{\eta,L}$, hence an element of $X^{(N)}_{\eta,0}(L)$, which is rationally equivalent to $Nx$ in $X_{\eta,L}$.  We perform this construction with the generic point $\gamma_1$ of $X_\eta$, which is defined over $L_1=K(X_\eta)$. This provides the desired rational map $\sigma$ satisfying condition \ref{item119juill}. The rational map  $\sigma$ satisfies   (\ref{item219juill}) by construction.
\end{proof}

  With the notations of Claim \ref{claim}, we now have the following corollary of Lemma \ref{prodiag}:
\begin{coro}\label{eqpourrestdiagcomp}  For   some  integers $\alpha,\,\beta,\,\gamma$ with $\alpha$ and $\beta$ odd, the two correspondences $\Gamma_N$ and $Z_N$ between $X_{\eta,0}^{(N)}$ and $X_{\eta}$ satisfy the following comparison:
\begin{eqnarray}\label{eqcompacorrespdelat} \alpha \Gamma_N\circ \sigma=\beta {Z}_{M}\circ \sigma +\gamma X_{\eta}\times h_{X_{\eta}}+ \Gamma'\,\,{\rm in}\,\,{\rm CH}^3(X_{\eta}\times X_{\eta}),
\end{eqnarray}
where $\Gamma'$ is a   codimension $3$ cycle  of $X_{\eta}\times X_{\eta}$ which is supported on  $W\times X_{\eta}$ for some proper closed algebraic subset $W\subset X_{\eta}$.
\end{coro}
\begin{proof} Indeed, we apply Lemma \ref{prodiag} that tells us that both cycles
 $\Gamma_N\circ \sigma$ and ${Z}_{M}\circ \sigma$ can be written, after restriction to $\underset{{U}\underset{{\rm open}}{\subset} X_{\eta}}{\underrightarrow{\rm lim}}{\rm CH}^3(U\times X_{\eta})$,  as combinations with integral coefficients of the cycles $\Delta_{X_{\eta}}$ and $X_{\eta}\times h_{X_{\eta}}$. Moreover, via the first projection, the cycle  $\Gamma_N\circ \sigma$ has degree $N$ over $X_{\eta}$,  the cycle  ${Z}_{M}\circ \sigma$ has degree $M$ over $X_{\eta}$, the diagonal $\Delta_{X_{\eta}}$ has degree $1$ over $X_{\eta}$ and the cycle $X_{\eta}\times h_{X_{\eta}}$ has degree $4$ over $X_{\eta}$. Writing
 \begin{eqnarray}\label{eqcompacorrespdelatfinrev}\Gamma_N\circ \sigma=a\Delta_{X_{\eta}}+bX_{\eta}\times h_{X_{\eta}}\,\,{\rm in}\,\,\underset{{U}\underset{{\rm open}}{\subset} X_{\eta}}{\underrightarrow{\rm lim}}{\rm CH}^3(U\times X_{\eta}),\\
 \nonumber
 Z_M\circ \sigma=c\Delta_{X_{\eta}}+dX_{\eta}\times h_{X_{\eta}}\,\,{\rm in}\,\,\underset{{U}\underset{{\rm open}}{\subset} X_{\eta}}{\underrightarrow{\rm lim}}{\rm CH}^3(U\times X_{\eta}),\end{eqnarray}
 we conclude, using the fact that both $M$ and $N$ are odd and by comparing the degrees in (\ref{eqcompacorrespdelatfinrev}), that $a$ is odd and $c$ is odd, from which (\ref{eqcompacorrespdelat}) follows by the localization exact sequence, since by (\ref{eqcompacorrespdelatfinrev}) we can take $\alpha=c,\,\beta=a$.
\end{proof}
We now conclude the proof of Proposition \ref{profinal}. The constructions above specialize from the generic fiber $X_{\eta}$ to the special fiber $\overline{X}$ of Theorem \ref{theototaro}.
The cycle ${Z}_M$ has a Fulton specialization to a cycle $$\overline{Z}_M\in {\rm CH}^3(\overline{X}^{(N)}_0\times \overline{X}),$$
which is effective by \cite[Lemme 2.10]{colliot}, and has degree $M$ over $\overline{X}^{(N)}_0$.

We  claim that   the rational map $\sigma$  of Claim \ref{claim}, specializes  well. We observe for this that, given a point $\overline{x}$ of $\overline{X}$ which is the specialization of a point $x$ of $X_\eta$, the specialization $C_{\overline{x}}$ of the  curve $C_x$   used in the proof of   Claim \ref{claim}  is a general cyclic double cover of a smooth conic ramified at $8$ points and $\overline{x}$ is a general point on it. The curve $C_{\overline{x}}$ is thus singular but it is smooth at $\overline{x}$ and it is  lci, hence  $\overline{x}$ defines  a Cartier divisor $\mathcal{O}_{C_{\overline{x}}}(D_{\overline{x}})$ on $C_{\overline{x}}$.   Serre duality shows that  $H^1(C_{\overline{x}},\mathcal{O}_{C_{\overline{x}}}(ND_{\overline{x}}-z))=0$ for any divisor $z\subset C_{\overline{x}}$ of degree $2$ supported on the regular locus $C_{\overline{x},{\rm reg}}$ of $C_{\overline{x}}$, so that the Cartier divisor $ND_{\overline{x}}$ is very ample on $C_{\overline{x},{\rm reg}}$.  We can then apply Bertini and conclude that there is a smooth member $z_{\overline{x}}$ of $|ND_{\overline{x}}|$, assuming the residue field of $\overline{x}$ is infinite. This also provides a smooth $0$-dimensional subscheme $z_{\overline{x}}$ of length $N$  of $\overline{X}$, hence an element of $\overline{X}^{(N)}_0$ defined over the same field as $\overline{x}$.  As $H^1(C_{\overline{x}},\mathcal{O}_{C_{\overline{x}}}(ND_{\overline{x}}))=0$, the linear system $|ND_x|$ on $C_x$ specializes to the  linear system $|ND_{\overline{x}}|$ on $C_{\overline{x}}$.
Applying these considerations to the generic point  of $\overline{X}$ gives the desired specialization $\overline{\sigma}$ of $\sigma$. This proves the claim.

As before, the existence of the specialized rational map $\overline{\sigma}$ inducing over an open set $\overline{U}\subset \overline{X}$ a morphism with value in $\overline{X}^{(N)}_0$ allows to define
the composed correspondences
$$\Gamma_N\circ \overline{\sigma},\, \overline{Z}_M\circ\overline{\sigma}\in {\rm CH}^3(\overline{U}\times \overline{X}).$$
As these correspondences are the Fulton specializations of the correspondences  (\ref{eqtrucdu26juill}), they  still satisfy the relation (\ref{eqcompacorrespdelat}).

 Using the  desingularization $\tau:\widetilde{\overline{X}}\rightarrow \overline{X}$, we lift the cycle $\overline{Z}_M$ to an effective cycle  $  \widetilde{\overline{Z}}_M$, and we lift similarly the specialized  rational map $\overline{\sigma}$ (keeping the same notation).  Using the fact that $\tau$ is ${\rm CH}_0$-universally trivial (see Theorem \ref{theototaro}), we conclude that formula (\ref{eqcompacorrespdelat}) still holds for the lifted specialized cycles, giving

\begin{eqnarray}\label{eqcompacorrespdelatspec}  \alpha \Gamma_N\circ \overline{\sigma}= \beta \widetilde{\overline{Z}}_M\circ \overline{\sigma} +\gamma \widetilde{\overline{X}}\times h_{\widetilde{\overline{X}}}\,\,{\rm in}\,\,\underset{{U}\underset{{\rm open}}{\subset} \widetilde{\overline{X}}}{\underrightarrow{\rm lim}}{\rm CH}^3(U\times \widetilde{\overline{X}}).
\end{eqnarray}

Note that, for the same reasons,  formula (\ref{item219juill}) also  holds for the lifted specialized cycles, giving the formula
\begin{eqnarray}\label{tructruc1907}  \Gamma_N\circ \overline{\sigma}= N\Delta_{\widetilde{\overline{X}}}\,\,{\rm in}\,\, \underset{{U}\underset{{\rm open}}{\subset} \widetilde{\overline{X}}}{\underrightarrow{\rm lim}}\,{\rm CH}^3(U\times \widetilde{\overline{X}}).
\end{eqnarray}

As in Totaro's argument and as was already used in the proof of Theorem \ref{theocrit}, we  deduce from (\ref{eqcompacorrespdelatspec}) the similar equality for the actions of the considered cycles on $2$-forms $\omega\in H^0(\widetilde{\overline{X}},\Omega^2_{\widetilde{\overline{X}}/\mathbb{F}})$. Here the field $\mathbb{F}$ has characteristic $2$, and using the fact that the numbers $\alpha $ and $\beta$ are odd, we finally conclude that, for any $\omega\in H^0(\widetilde{\overline{X}},\Omega^2_{\widetilde{\overline{X}}/\mathbb{F}})$,
\begin{eqnarray}\label{eqcompacorrespdelatspeconforms}  \overline{\sigma}^*( [\Gamma_N]^*\omega)=  \overline{\sigma}^*([\widetilde{\overline{Z}}_M]^*\omega)  \,\,{\rm in}\,\,H^0(\widetilde{\overline{X}},\Omega^2_{\widetilde{\overline{X}}/\mathbb{F}}).
\end{eqnarray}
Finally, as $N$ is odd, it follows from (\ref{tructruc1907}), by letting cycle classes  act on  $2$-forms, that for any $\omega\in H^0(\widetilde{\overline{X}},\Omega_{\widetilde{\overline{X}}/\mathbb{F}}^2)$,
\begin{eqnarray} \label{tructructruc1907} \overline{\sigma}^*([\Gamma_N]^*\omega)=N\omega=\omega\,\,{\rm in}\,\,H^0(\widetilde{\overline{X}},\Omega_{\widetilde{\overline{X}}/\mathbb{F}}^2).
\end{eqnarray}
Using (\ref{tructructruc1907})  and Lemma \ref{pro2form}, we get  that  $\overline{\sigma}^* $ induces an isomorphism between
$H^0( \widetilde{\overline{X}}^{(N)}_0, \Omega^2_{\widetilde{\overline{X}}^{(N)}_0/\mathbb{F}})$ and $H^0(\widetilde{\overline{X}},\Omega^2_{\widetilde{\overline{X}}/\mathbb{F}})$. It follows then from (\ref{eqcompacorrespdelatspeconforms}) that, for any $\omega\in H^0(\widetilde{\overline{X}},\Omega^2_{\widetilde{\overline{X}}/\mathbb{F}})$,
\begin{eqnarray}\label{eqcompacorrespdelatspeconformssurxN}   [\Gamma_N]^*(\omega)=  [\widetilde{\overline{Z}}_M]^*(\omega)  \,\,{\rm in}\,\,H^0(\widetilde{\overline{X}}^{(N)}_0,\Omega^2_{\widetilde{\overline{X}}^{(N)}_0/\mathbb{F}}).
\end{eqnarray}
As $\widetilde{\overline{Z}}_M$ is effective of degree $M<N$ over $\widetilde{\overline{X}}^{(N)}_0$, and $\widetilde{\overline{X}}$ has a nonzero $2$-form, we finally get a contradiction by the   rank argument used to prove Theorem \ref{theocrit}, \ref{itsit2}. \end{proof}

 We can prove the same result for quartic fourfolds by a Totaro-Koll\'ar specialization to a double cover $\overline{X}$ of a quadric in characteristic $2$ such that a desingularization $\widetilde{\overline{X}}$  has a nonzero $3$-form (or double covers of $\mathbb{P}^4$ ramified along a sextic or octic hypersurface). We have now to use  Theorem \ref{theocrit}, \ref{itsit3}, which deals with the rank of traces of $3$-forms, instead of $2$-forms. The only extra proof needed is the analogue of Lemma \ref{pro2form}, now with $l=3,\,n=4$.  This is left to the reader.

\section{The case of del Pezzo surfaces \label{seccomplement}}
We give  here for completeness the proof of Theorem \ref{theocompletdp} (which is established in \cite{colliot}, \cite{voisindelpezzo}, only when $d_S\leq 3$).  We will only describe the strategy of the proof, which follows closely \cite{voisindelpezzo}, to which we refer for more details. In practice, one should  be a little  careful in checking  case by case a certain technical statement (see Remark \ref{reamresteaprouver} below).
Let $\mathcal{O}_S(1):=K_S^{-1}$ be the anticanonical line bundle of $S$ and let $h_S:=c_1(K_S)^2\in{\rm CH}_0(S)$. This is a $0$-cycle on $S$ of degree $d_S$. One has by Hirzebruch-Riemann-Roch and Kodaira vanishing
\begin{eqnarray}\label{eqHRR}h^0(S,\mathcal{O}_S(l))=1+\frac{(l^2+l)d_S}{2}.\end{eqnarray}

In order to prove Theorem \ref{theocompletdp}, it suffices to prove the following
\begin{prop} \label{propaprouver} There exists an integer $N_1$ with the following property: for any smooth del Pezzo surface $S$ over a field $K$ of characteristic $0$, and any   effective $0$-cycle $z\in{\rm CH}_0(S)$  of degree $d> N_1$, there exists  $\gamma \in\mathbb{Z}$ such that
 the  $0$-cycle $z':=\pm z+\gamma h_S$ is effective   of degree strictly smaller than $d$.
\end{prop}
Let us first  show how Proposition \ref{propaprouver} implies Theorem \ref{theocompletdp}. This argument already appears in \cite{voisindelpezzo}. From proposition \ref{propaprouver}, we deduce by induction on $d={\rm deg}\,z$ that any effective $0$-cycle on $S$ can be written modulo rational equivalence as
\begin{eqnarray}\label{eqpourcyclereddege} z=\pm z'+\gamma h_S\,\,{\rm in}\,\,{\rm CH}_0(S),\end{eqnarray}
where $\gamma\in \mathbb{Z}$ and $z'$ is an effective $0$-cycle of degree $\leq N_1$. Next by boundedness of del Pezzo surfaces, there exists an integer $N\geq N_1$ such that for any effective zero-cycle $z'$ on $S$ of degree $\leq N_1$, the  $0$-cycle $z''=\gamma' h_S-z'$ is both of degree $\leq N$ and  effective, for some number $\gamma'\in \mathbb{Z}$. We then conclude from (\ref{eqpourcyclereddege}) that any effective $0$-cycle on $S$ can be written modulo rational equivalence as
\begin{eqnarray}\label{eqpourcyclereddege1} z= z''+\gamma'' h_S\,\,{\rm in}\,\,{\rm CH}_0(S),\end{eqnarray}
where $z''$ is effective of degree $\leq N$.
Finally (\ref{eqpourcyclereddege1}) implies  Theorem \ref{theocompletdp} with the same integer $N$.
Indeed, any $0$-cycle $z$ can be written as \begin{eqnarray}\label{eqpourzzunzdeux} z=z_1-z_2\,\,{\rm  in}\,\, {\rm CH}_0(S)\end{eqnarray}  for some effective $0$-cycles $z_1$ and $z_2$. Given $z_2$, for some large enough  positive integer $\alpha$, $\alpha h_S-z_2$ is effective. Indeed, we choose a smooth projective  curve $C\subset X$ in some linear system $|-l K_S|$ supporting $z_2$. The curve $C$ also supports a multiple $\beta h_S$ of $h_S$, and by Riemann-Roch on $C$, we get that for $\beta'$ large enough, the $0$-cycle
$\beta'w-z_2$ is effective on $C$, where $w$ is any $0$-cycle on $C$ of class $\beta h_S$.
It follows that (\ref{eqpourzzunzdeux}) becomes
\begin{eqnarray}\label{eqpourzzunzdeux1} z=z_{\rm eff}+\delta h_S \,\,{\rm in}\,\,{\rm CH}_0(S)\end{eqnarray}
for some effective $0$-cycle $z_{\rm eff}$ on $S$. We can then apply (\ref{eqpourcyclereddege1}) to $z_{\rm eff}$ and conclude that any $0$-cycle $z$ can be written as
$$ z= z''+\gamma'' h_S\,\,{\rm in}\,\,{\rm CH}_0(S),$$
where $z''$ is effective of degree $\leq N$.
If ${\rm deg}\,z\geq N$, then $\gamma''\geq 0$ and $z$ is effective, proving Theorem \ref{theocompletdp}.

\begin{proof}[Proof of Proposition \ref{propaprouver}]
As we are over a field $K$ of characteristic $0$, an effective $0$-cycle of degree $d$ on $S$ is also a $K$-point $z\in S^{(d)}(K)$. As proved in  \cite[Lemme 2.10]{colliot} and  used repeatedly in \cite{voisindelpezzo} and in this paper, the Fulton specialization of cycles  preserves effectivity. It thus suffices to prove Proposition \ref{propaprouver} when $z\in S^{(d)}$ is generic (defined over the function field  $L:=K(S^{(d)})$), as the result for any $z$  then follows by specialization. The proof is by induction on the degree $d={\rm deg}\,z$.  Let $l$ be the unique integer such that

\begin{eqnarray}\label{eqineq} h^0(S,\mathcal{O}_S(l))< {\rm deg}\,z\leq h^0(S,\mathcal{O}_S(l+1)).
\end{eqnarray}
Assume first that (*) ${\rm deg}\,z\leq h^0(S,\mathcal{O}_S(l+1))-2$.

Then
$(l+1)^2h_S-z$ is effective by \cite[Lemma 3.6]{voisindelpezzo}. As ${\rm deg}\,(l+1)^2h_S=(l+1)^2d_S$, it follows that we can assume,   replacing
$z$ by $(l+1)^2h_S-z$ if necessary, that
\begin{eqnarray} \label{eqinedimres}{\rm deg}\,z \leq d_S\frac{(l+1)^2}{2}.\end{eqnarray}
Note that,   as our argument is  by induction on the degree $d={\rm deg}\,z$, we only have  to study the case where ${\rm deg}\,z ={\rm deg}\,( (l+1)^2h_S-z)= d_S\frac{(l+1)^2}{2}$, so the left inequality (\ref{eqineq}) still holds.

The generic effective $0$-cycle $z$ defines as well a reduced subscheme
$Z\subset S_L$ of length $d$. Using the left strict inequality in (\ref{eqineq}), we get that $H^1(S_L,\mathcal{I}_Z(l))\not=0$, hence by Serre duality on $S_L$, using $K_{S_L}=\mathcal{O}_{S_L}(-1)$,
$${\rm Ext}^1(\mathcal{I}_Z(l+1),\mathcal{O}_{S_L})\not=0.$$
A nonzero  extension class
$e\in {\rm Ext}^1(\mathcal{I}_{S_L}(l+1),\mathcal{O}_{S_L})$ provides a rank $2$ coherent sheaf
$E$ on $S_L$, which fits in an extension
\begin{eqnarray}\label{eqnou*}0\rightarrow \mathcal{O}_{S_L}\rightarrow E\rightarrow \mathcal{I}_{Z}(l+1)\rightarrow 0.\end{eqnarray}
Using the genericity of $z$, the coherent sheaf $E$ is in fact a vector bundle for a general choice of extension $e$. Furthermore it  has a section $s$ vanishing exactly along $Z$, which is given by the map $\mathcal{O}_{S_L}\rightarrow E$ in (\ref{eqnou*}). It follows that
$$z=c_2(E)\,\,{\rm in}\,\,{\rm CH}_0(S_L).$$
 Furthermore $E$  satisfies by (\ref{eqnou*})
\begin{eqnarray}\label{eqnounou*}h^0(S_L,E)=1+h^0(S_L,\mathcal{I}_Z(l+1)).\end{eqnarray}
Using (\ref{eqinedimres}) and (\ref{eqHRR}), we get from (\ref{eqnounou*}) that
\begin{eqnarray}\label{eqineqpourhzeroE} h^0(S,E)\geq 2+d_S\frac{(l+1)}{2}.
\end{eqnarray}
We then conclude using the genericity of $Z$ that, if $\frac{(l+1)}{2}\geq 2$, that is, if $l\geq 3$, there exists a section $s'$  of $E$  vanishing along a  subscheme $Z_0\subset S$ of class $h_S$.
If the vanishing locus of $s'$ has dimension $0$, then
the zero-cycle $V(s')-h_S$ is effective and rationally equivalent to $c_2(E)-h_S=z-h_S$. As the subschemes $Z_0$ are not generic, it is however not clear that there exists a section $s'$ vanishing along $Z_0$ and with $0$-dimensional vanishing locus. We overcome this by seeing $Z_0$ as the specialization of a generic $0$-dimensional subscheme
$Z_{0,{\rm gen}}\subset S_L$ of length $d_S$. We prove then the existence of a section $s'$ vanishing along $Z_{0,gen}$ and with $0$-dimensional vanishing locus. Thus $z-Z_{0,{\rm gen}}$ is effective, and by specialization $z-h_S$ is effective.

 In conclusion,  assuming (*), the induction step for Proposition \ref{propaprouver} is proved if  $l\geq 3$, hence if $d> h^0(S,\mathcal{O}_S(3))$. As $h^0(S,\mathcal{O}_S(3))\leq 55$, the induction step works for $d\geq 56$, assuming (*).

 In the remaining cases where
\begin{eqnarray}\label{eqremainingcase}  d={\rm deg}\,z= h^0(S,\mathcal{O}_S(l+1))-1\,\, {\rm or}\,\,{\rm deg}\,z= h^0(S,\mathcal{O}_S(l+1)),\end{eqnarray}
  we do the following trick.  Choosing a  reduced subscheme $Z_0\subset S_L$ of class $z_0=h_S$ and not intersecting $Z$, we consider the $0$-dimensional subscheme  $Z'=Z\cup Z_0\subset S_L$ which is reduced  and effective of  degree \begin{eqnarray}\label{eqinehzerodeg}d+d_S>h^0(S_L,\mathcal{O}_{S_L}(l+1))\end{eqnarray}  (we can assume here $d_S\geq 4$ by Theorem \ref{theorappeldp}).
  Using  inequality (\ref{eqinehzerodeg}), we now apply the same strategy of constructing a vector bundle $E$ which fits in an extension
  $$0\rightarrow \mathcal{O}_{S_L}\rightarrow E\rightarrow \mathcal{I}_{Z'}(l+2)\rightarrow 0,$$
  and has as above a section $s$ vanishing exactly along $Z'=Z\cup Z_0$, where $Z_0$ represents the class $h_S$.
  One has $$h^0(S_L,E)=1+h^0(S_L,\mathcal{I}_{Z'}(l+2))
  \geq 1+h^0(S_L,\mathcal{O}_{S_L}(l+2))-h^0(S_L,\mathcal{O}_{S_L}(l+1))-d_S$$
  $$= (l+1)d_S+1.$$
  It follows that,  once $d_S(l+1)\geq 4d_S$, that is,
  \begin{eqnarray}
  \label{eqineqpourl2} l+1\geq 4,
  \end{eqnarray}
  the vector bundle $E$ has a section $s'$   vanishing along $Z'_0\cup Z''_0$, where $Z'_0,\, Z''_0$ are subschemes in general position representing the class $h_S$.  In this construction, the subscheme $Z'=Z\cup Z_0$, or rather the corresponding  element
   $z'\in S_L^{(d+d_S)}$, is not generic in $S_L^{(d+d_S)}$, and similarly the subschemes $Z'''_0=Z'_0\cup Z''_0$, or rather the corresponding elements $z'''_0=z'_0+z''_0\in S_L{(2d_S)}$, are not generic, but we can replace them by the generic points $z'_{\rm gen}\in S_L^{(d+d_S)}$, resp. $Z'''_{0,{\rm gen}}\in  S_L^{(2d_S)}$ which specialize to $z'$, resp. $z'''_0$.
     One checks that, when the zero-dimensional subschemes above are generic, the generic section $s'$ of $E$ vanishing along  $Z'''_{0,{\rm gen}}$ has a zero-dimensional vanishing locus, hence  $z'_{\rm gen}-z'''_{0,{\rm gen}}$ is effective. It follows by specialization, applying again \cite[Lemme 2.10]{colliot},    that
   $$z'-z'_0-z''_0=z+z_0-z'_0-z''_0=z-h_S\in {\rm CH}_0(S_{L})$$ is effective.

  The inequality (\ref{eqineqpourl2}), together with (\ref{eqremainingcase}) shows that in the situation (\ref{eqremainingcase}), the induction step of Proposition \ref{propaprouver} works once $l\geq 3$. When $l\leq 2$ and (\ref{eqremainingcase}) holds, we have
  $$d\leq h^0(S,\mathcal{O}_S(3))\leq h^0(\mathbb{P}^2,\mathcal{O}_{\mathbb{P}^2}(9))=55.$$
  Proposition \ref{propaprouver} is thus proved with $N_1=56$.
\end{proof}
\begin{rema}\label{reamresteaprouver} {\rm  In the sketch of proof above, what needs an extra-checking is the proof that, when the $0$-cycles $z$ (or $z'$) and $z_0$ (or $z'''_0$) are generic,  one can choose the section $s'$ as above to have a zero-dimensional vanishing locus. This proof is done carefully in the case of cubic surfaces in \cite{voisindelpezzo}.}
\end{rema}

\end{document}